\documentclass[10pt]{article}
\usepackage{amssymb, amsmath, amsthm, hyperref, color, geometry, fancyhdr, mathtools, enumerate, scalerel, color, accents, fancybox, rotating, colortbl, stmaryrd,  cleveref, setspace, tikz, tikz-3dplot, graphics, comment}

  \usepackage{mathabx}
\usepackage[all,2cell]{xy}
   \UseAllTwocells
   \SilentMatrices
\usepackage[dvips]{epsfig}

\newcommand{\bbullet}{
\begin{tikzpicture}
  \draw[fill=black] circle (0.55ex);
\end{tikzpicture} 
}

\usetikzlibrary{arrows, decorations.markings, decorations, patterns,
positioning, decorations.pathreplacing, decorations.pathmorphing,
intersections, fit}
\tikzset{->-/.style={decoration={markings, mark=at position .5 with {\arrow{>}}},postaction={decorate}}}
\tikzset{-<-/.style={decoration={markings, mark=at position .5 with {\arrow{<}}},postaction={decorate}}}

\usetikzlibrary{arrows, decorations.markings, decorations, patterns,
positioning, decorations.pathreplacing, decorations.pathmorphing, external}
\tikzset{->-/.style={decoration={markings, mark=at position .5 with {\arrow{>}}},postaction={decorate}}}
\tikzset{-<-/.style={decoration={markings, mark=at position .5 with {\arrow{<}}},postaction={decorate}}}

\newcommand{\projective}[1]{\ensuremath{P_{#1}}}

\newcommand{\Kh}{\ensuremath{\mathrm{Kh}}}
\newcommand{\CKh}{\ensuremath{\mathrm{CKh}}}
\newcommand{\SKh}{\ensuremath{\mathrm{SKh}}}
\newcommand{\VKh}{\ensuremath{\mathrm{VKh}}}
\newcommand{\PKh}{\ensuremath{\mathrm{PKh}}}
\newcommand{\TKh}{\ensuremath{\mathrm{TKh}}}
\newcommand{\imagesfolder}{.}
\newcommand{\NB}[1]{\ensuremath{\vcenter{\hbox{#1}}}}

\newcommand{\mapH}{\ensuremath{\eta}}

\newcommand{\jw}{\mathtt{V}}
\newcommand{\pjw}{\mathtt{P}}
\newcommand{\hjw}{\mathtt{H}}

\newcommand{\hl}{\ensuremath{\mathit{h}}}
\newcommand{\el}{\ensuremath{\mathit{e}}}
\newcommand{\fl}{\ensuremath{\mathit{f}}}

\newcommand{\tikztilting}[5]{
  \begin{scope}[xshift=#1 cm, yshift=#2 cm, xscale=#3, yscale=#4]
    \filldraw[fill=pink, draw = black] (-1.5, -0.6) rectangle (1.5, 0.6) coordinate[pos=0.5] (A);
    \begin{scope}
      \foreach \i in {-0.5, -0.4, ..., 0.6}{
      \draw[thin] (-1.5, \i) -- ( -1.3, \i );
      }
\end{scope}
    \node at (A) {#5};
\end{scope}}

\normalfont\upshape

\theoremstyle{definition}
\newtheorem{thm}{Theorem}[section]
\newtheorem{cor}[thm]{Corollary}

\newtheorem{lem}[thm]{Lemma}
\newtheorem{rem}[thm]{Remark}

\newtheorem{defn}[thm]{Definition}
\newtheorem{example}[thm]{Example}

\newtheorem*{thm*}{Theorem}

\numberwithin{equation}{section}

\def\N{{\mathbb N}}

\def\Z{{\mathbb Z}}

\renewcommand{\to}{\rightarrow}

\newcommand{\id}{{\rm id}}

\newcommand{\End}{{\rm End}}

\def\dif{\partial}

\def\lra{{\longrightarrow}}
   
\def\dmod{{\mathrm{\mbox{\textrm{-}}mod}}}

\def\Id{\mathrm{Id}}
\def\mc{\mathcal}

        \def\shuffle{\,\raise 1pt\hbox{$\scriptscriptstyle\cup{\mskip
               -4mu}\cup$}\,}
\def\udmod{{\mathrm{\textrm{-}\underline{mod}}}}

\newcommand{\refequal}[1]{\xy {\ar@{=}^{#1}
(-1,0)*{};(1,0)*{}};
\endxy}

\newcommand{\mH}{\mathrm{H}} 
\newcommand{\Ul}{{\mathsf U}}

\newcommand{\ul}{{\mathsf u}}

\newcommand{\dotnewtoni}[1][i]{\ensuremath{\textcolor[rgb]{0,0,0.6}{\bbullet_{#1}}}}

\DeclareMathOperator{\Dist}{Dist}

\newcommand{\TL}{\mathsf{TL}}

\DeclareMathOperator{\tl}{tl}

\renewcommand{\bar}[1]{\overline{#1}}

\def\JS#1{\textcolor[rgb]{0.40,0.00,0.90}{\bf [JS: #1]}}

\title{Colored Khovanov homology theories at prime roots of unity}
\author{You Qi, Louis-Hadrien Robert, Joshua Sussan and Emmanuel Wagner}

 \hypersetup{
    pdftoolbar=true,        
    pdfmenubar=true,        
    pdffitwindow=false,     
    pdfstartview={FitH},    
    pdftitle={On some colored Khovanov homology theories at prime
roots of unity},    
    pdfauthor={You Qi, Louis-Hadrien Robert, Joshua Sussan and Emmanuel Wagner},     
    pdfsubject={On some colored Khovanov homology theories at prime
roots of unity},   
    pdfcreator={You Qi, Louis-Hadrien Robert, Joshua Sussan and Emmanuel Wagner},   
    pdfproducer={You Qi, Louis-Hadrien Robert, Joshua Sussan and Emmanuel Wagner}, 
    pdfkeywords={}, 
    pdfnewwindow=true,      
    colorlinks=true,       
    linkcolor=darkgray,          
    citecolor=teal,        
    filecolor=magenta,      
    urlcolor=violet,          
    linkbordercolor=red,
    citebordercolor=teal,
    urlbordercolor=violet,  
    linktocpage=true
    }
\date{\today}

\begin{document}

\maketitle

\begin{abstract}
We equip Khovanov's colored $\mathfrak{gl}(2)$-homology groups with a degree-2 derivation. Over a field of characteristic $p>2$, the derivation specializes to a $p$-nilpotent differential, which leads to a categorification of colored $\mathfrak{gl}(2)$-homology at a $2p$th root of unity.
\end{abstract}

\setcounter{tocdepth}{2} \tableofcontents

\section{Introduction}
Khovanov homology \cite{KhJones} categorifies the Jones polynomial where the quantum parameter is assumed to be generic.  This was extended to a categorification of the colored Jones polynomial in \cite{Khcolored}.
Beliakova and Wehrli \cite{BW} showed that this construction could be done over $\mathbb{Z}$ (with 2 inverted).  Von Merkl \cite{VM} showed that various constructions in \cite{Khcolored} are isomorphic. These categorifications all yield finite-dimensional homology theories.  We will not discuss various interesting categorifications of the colored Jones polynomial which have infinite-dimensional homologies.

The WRT 3-manifold invariant may be defined as a sum of colored Jones polynomials evaluated at a root of unity. 
An approach to categorifying this invariant was introduced by Khovanov \cite{Hopforoots} who suggested looking for $p$-differentials on algebraic structures defined over a field of characteristic $p$.  These are derivations $\dif$ such that $\dif^p=0$.

In \cite{QRSW2, QRSW3, QRSW4}, the authors studied $\mathfrak{sl}(2)$-actions on equivariant Khovanov--Rozansky homology.  Working over a field of characteristic $p$, one obtains a $p$-differential $\dif=f$ where $f$ is a generator of $\mathfrak{sl}(2)$.

In this paper we show the following. 
\begin{thm*}
    Colored Khovanov homology (colored by tilting modules over the distribution algebra for $\mathfrak{sl}_2$ in characteristic $p$) may be equipped with a $p$-differential yielding a categorification of the colored Jones polynomial at a $2p$th root of unity.  In particular, when the colors of the components are less than $p$, we obtain a categorification of the standard colored Jones polynomial at $2p$th roots of unity. This colored homology theory is functorial.
\end{thm*}

An approach to a categorification of the (colored) Jones polynomial using $\mathfrak{gl}(-2)$-homology was constructed in \cite{QiSussanLink, QRSW1} and is orthogonal to this work.

\subsubsection{Acknowledgements}
The authors are grateful to Alexis Guérin, Mikhail Khovanov, Felix Roz, and Paul Wedrich for helpful comments.

Y.~Q. is partially supported by the Simons Foundation and the National Science Foundation (DMS-2401376).  J.~S.{} is
partially supported by the NSF grant DMS-2401375 and a Simons Foundation Travel Support Grant. E.~W.{} is partially supported by the ANR project Caget (ANR-25-CE40-4162). 

This material is based upon work supported by the National Science Foundation under Grant No. DMS-1929284 while the author was in residence at the Institute for Computational and Experimental Research in Mathematics in Providence, RI, during the ``Categorification of WRT invariants of 3-manifolds" Collaborate@ICERM program.

The authors used GPT-5.6 Sol to proofread the manuscript.

\section{Background}

\subsection{\texorpdfstring{$p$}{p}-DG  theory} \label{sec-DG-Frob} 
For the sake of completeness, we review the necessary $p$-DG theory for link homology in this section.  See \cite{QRSW1} for more details.

Let $\Bbbk$ be a field of characteristic $p>2$. An unadorned tensor product ``$\otimes$'' will be understood to be taken over $\Bbbk$.

Let $H = \Bbbk[\dif]/(\dif^p)$ be the graded truncated polynomial algebra generated
by a degree $2$ generator $\dif$. We endow $H$ with a graded Hopf algebra structure by defining comultiplication $\Delta:H\to H\otimes H$, counit $\epsilon: H \to \Bbbk$ and antipode $S: H\to H$ by the following formulas:
\begin{equation}\label{eqn-Hopf-algebra-H-prime}
    \Delta (\dif) = \dif \otimes 1 + 1\otimes \dif, \qquad 
    \epsilon(\dif)=0, \qquad \text{and}\qquad
    S(\dif)=-\dif.
  \end{equation}

Since $H$ is a graded Hopf algebra, there is a natural grading-reversing duality on $H$-modules, which we denote by $(\mbox{-})^*$. Given an $H$-module $M$ and $\dif \in H$, $\dif$ acts on the dual vector space $M^*$ of $M$ by
\begin{equation}\label{eqn-dual-rep}
(\dif \cdot f)(v):=-f(\dif \cdot v),
\end{equation}
for any $f\in M^*$ and $v\in M$.

\subsubsection{$p$-Complexes} Since $H$  is graded
commutative and cocommutative, its category of graded modules, denoted
$H\dmod$, is a symmetric monoidal abelian
category, where the monoidal structure is given by the usual vector
space tensor product $\otimes$. 
A graded module over $H$ will be referred to as a \emph{$p$-complex}
since, when $p=2$, such an object reduces to the usual notion of a
(cohomological) complex in characteristic two.

Note that $H$ being a $\Bbbk$-vector space, any $p$-complex is in
particular a $\Bbbk$-vector space. Hence when dealing with $p$-complexes, we
are working in characteristic $p$.

Graded $H$-modules constitute a Krull--Schmidt category, with
indecomposable $p$-complexes over $\Bbbk$ classified as follows. Set
\begin{equation}
  U_i:=H/(\dif^{i+1}) ,\quad \quad 0\leq i \leq p-1,
\end{equation}
so that $U_i$ has dimension $i+1$. Given any $a\in \Z$, denote by
$q^aU_i$ the module $U_i$ with grading shifted up by $a$. Then
$q^aU_i$ is concentrated in degrees $a, a+2,\dots, a+2i$:
\begin{equation}
q^aU_i=\left(
    \NB{\tikz[scale = 1.5 ]{
        \node (Fa) at (0,0) {$\overset{a}{\Bbbk}$};
        \node (Fa2) at (1,0) {$\overset{a+2}{\Bbbk}$};
        \node (dots) at (2,0) {$\cdots$};
        \node (Fa2i) at (3,0) {$\overset{a+2i}{\Bbbk}$};
        \draw[double equal sign distance] (Fa) -- (Fa2);
        \draw[double equal sign distance] (Fa2) -- (dots);
        \draw[double equal sign distance] (dots) -- (Fa2i);
      }}\right) \ .
\end{equation}
The objects $q^a U_i$, $a\in \Z$ and $i\in \{0,\dots, p-1\}$
constitute a full list of indecomposable graded $H$-modules.

As a finite-dimensional Hopf algebra, $H$ is graded Frobenius
\cite{LaSw}, meaning that the full subcategory of graded injective
modules in $H\dmod$ coincides with that of graded projectives. The
\emph{(graded) stable module category} $H\udmod$ is the quotient of
$H\udmod$ by the class of graded projective-injective modules. It is a
triangulated category \cite[Theorem 2.6]{Hap88}. For any graded
$H$-module $M$, the \emph{homological shift functor} $t$ is defined
by taking a graded injective embedding $\jmath_M$:
\begin{equation}
  0 \lra M \stackrel{\jmath_M}{\lra} I_M \lra \mathrm{Coker}(\jmath_M) \lra 0, 
\end{equation}
and declaring $tM:=\mathrm{Coker}(\jmath_M)$. The inverse functor
$t^{-1}$ is constructed by taking a graded projective cover $\rho_M$:
\begin{equation}
  0 \lra \mathrm{Ker}(\rho_M) \lra P_M \stackrel{\rho_M}{\lra} M \lra 0,
\end{equation}
and setting $t^{-1}M:=\mathrm{Ker}(\rho_M)$.

The symmetric monoidal structure on $H\dmod$ descends to an exact
symmetric monoidal structure on $H\udmod$, which we still denote by
$\otimes_{\Bbbk}$.
This follows from the fact that, given an $H$-module $M$ of
dimension $n$, $H\otimes_{\Bbbk} M$ is free of rank $n$ (see, for instance,
\cite[Proposition 2.1]{QYHopf} for details).

Similarly, the category $H\udmod$ inherits exact degree shift
functors from $H\dmod$, which we will still write as $q^i(\mbox{-})$ for any
$i \in \Z$. One has functorial-in-$M$ isomorphisms
\begin{equation}
  q^iM \cong M\otimes (q^i\Bbbk) \cong (q^i\Bbbk) \otimes M.
\end{equation}

Let $U_{p-2}$ be the $(p-1)$-dimensional $p$-complex which is graded
isomorphic to $H/(\dif^{p-1})$. Then there is an isomorphism
\begin{equation}\label{eqn-V-p-2}
  tM \cong q^{2-2p}U_{p-2}\otimes M,
\end{equation}
functorial in $M$. This is because, as alluded to above, $H\otimes M$
is free of rank $n$, and thus the embedding $M \lra H\otimes M$,
$m\mapsto \dif^{p-1}\otimes m$ can be used as $\jmath_M$. Furthermore,
this choice of $\jmath_M$ is clearly functorial.

Applying the shift functor $t$ twice to the module $\Bbbk$ shows that
\begin{equation}
  t^2\Bbbk\cong q^{-2p} \Bbbk,
\end{equation}
and thus the functor isomorphism
$t^2(\mbox{-})\cong q^{-2p}(\mbox{-})$. This follows from the easily
verified
fact that the tensor product $q^{4-2p}U_{p-2}\otimes U_{p-2}$ decomposes as $U_0 \oplus \bigoplus_{i=0}^{p-3} q^{4-2p+ 2i}H$.

\subsubsection{Slash cohomology} Associated to a $p$-complex $M$, one can
define its {\sl slash cohomology} as in \cite[Section 2.1]{KQ}. This
is an analogue of the usual homology functor for chain complexes.

For each $0\le k \le p-2$, form the graded vector space
\[ \mH_{/k}(M) = \dfrac{\mathrm{Ker}(\dif^{k+1}_M)}{\mathrm{Im}
  (\dif^{p-k-1}_M)+\mathrm{Ker} (\dif^{k}_M)}.\] The original
$\Z$-grading on $M$ gives a decomposition
\[ \mH^{\bullet}_{/k}(M) = \bigoplus_{i\in \Z} \mH^{ i }_{/k}(M) .\]
 Define the {\sl
slash cohomology} of $M$ as
$\displaystyle{ \mH^{\bullet }_{/}(M) =  \bigoplus_{k=0}^{p-2} \mH^{\bullet }_{/k}(M)}$  and
let
$\displaystyle{\mH^{i}_/(M) :=
  \bigoplus_{k=0}^{p-2}\mH_{/k}^{i}(M)}$. We have the decompositions
\begin{equation}
  \mH^{\bullet}_/(M) = \bigoplus_{i\in \Z} \mH^{i}_/(M)  = \bigoplus_{k=0}^{p-2}
  \mH^{\bullet}_{/k}(M)=\bigoplus_{i\in \Z}\bigoplus_{k=0}^{p-2} \mH_{/k}^{ i}(M).
\end{equation}
The differential $\dif_M$ induces an operator also denoted $\dif_M$, of bidegree $(-1,2)$ on 
the bigraded $\Bbbk$-vector space $\mH_/^{\bullet}(M)$.

Forgetting the $k$-grading gives us a graded vector space
$\mH^{\bullet}_/(M)$ with differential $\dif_M$, which we can also
view as a graded $H$-module.  There is a decomposition
\begin{equation}\label{eqn-decomp-of-M-into-cohomology-and-free}
  M\cong \mH_/^{\bullet}(M)\oplus P(M)
\end{equation}
in the abelian category of $H$-modules, where $P(M)$ is a maximal
projective direct summand of $M$. It follows that $\mH_/^{\bullet}(M)$
is isomorphic to $M$ in the stable category $H\udmod$.

In particular, we have
\begin{equation}
  \mH_/^\bullet(q^aU_i) = 
  \begin{cases}
    q^aU_i , & a\in \Z, \quad i=0,\dots, p-2, \\
    0 & i=p-1 .
  \end{cases}
\end{equation}
The slash cohomology group $\mH_/^{\bullet}(M)$, viewed as an
$H$-module, does not contain any direct summand isomorphic to a graded
free $H$-module.

The assignment $M\mapsto \mH_{/}^\bullet(M)$ is functorial in $M$ and
can be viewed as a functor $H\dmod \lra H\udmod$ or as a functor
$H\udmod \lra H\udmod$. The latter functor is then isomorphic to the
identity functor.

\subsubsection{A theorem of Rickard} For later use, we will record an
alternative description of $H\udmod$ that is due to Rickard \cite{Ri}.

For a graded algebra $A$, we will let $\mc{C}(A,d)$ stand for the
usual graded homotopy category of chain complexes ($d^2=0$) of
$A$-modules. Likewise, we let $\mc{D}(A,d)$ stand for the usual
derived category of $A$. By standard convention, we let
$\mc{C}^b(A,d)$ and $\mc{D}^b(A,d)$ be the usual {\sl bounded}
homotopy and derived categories of $A$. These categories are
triangulated, with homological degree shifts denoted by $t_d^{\pm1}$ 
in contrast to $t^{\pm1}$ 
on $H\udmod$ recalled earlier.

For a Hopf algebra $H$, $\mc{D}^b(H,d)$ is equipped with a monoidal
structure as follows. If $M=(\oplus_{i}M^i,d_M)$ and
$N=(\oplus_{i} N^i,d_N)$ are complexes of $H$-modules, then set
$M\otimes N$ to be the complex whose homogeneous degree $i$ term and
differential action are given by
\begin{equation}
  (M\otimes N)^i:=\bigoplus_{k} M^k\otimes N^{i-k},\quad d_{M\otimes N} (x\otimes y):= d_M(x)\otimes y +(-1)^k x\otimes d_N(y)
\end{equation}
for any $x\in M^k$ and $y\in N^{i-k}$.

A complex of graded $H$-modules is called \emph{perfect} if it is
quasi-isomorphic to a bounded complex of projective $H$-modules. The
full subcategory of perfect complexes in $\mc{D}(H,d)$ will be denoted
$\mc{P}_H$, and it is idempotent closed or \emph{thick} \cite[Theorem
2.1]{Ri}. Furthermore, $\mc{P}_H$ is a monoidal ideal in $\mc{D}(H,d)$
again because $H\otimes M$ and $M\otimes H$ are projective for any
$H$-module $M$.
One can thus form the Verdier quotient category
$\mc{D}^b(H,d)/\mc{P}_H$, which inherits the triangulated monoidal
category structure from that of $\mc{D}^b(H,d)$.

\begin{thm}\label{thm-Rickard}
  There is an equivalence of triangulated monoidal categories
  \begin{equation}
    \mc{R}: H\udmod \lra   \mc{D}^b(H,d)/\mc{P}_H 
  \end{equation}
\end{thm}
\begin{proof}
  This is \cite[Theorem 2.1]{Ri}, which is only stated for the
  ungraded situation, but the proof there adapts verbatim to the
  graded situation. We only need to remark that the Rickard functor
  $\mc{R}$ sends an $H$-module $M$ into the complex whose only nonzero
  term is $M$ itself supported in homological degree zero. This is
  clearly compatible with the monoidal structures on $H\udmod$ and
  $\mc{D}^b(H,d)$.
\end{proof}

\begin{rem}
  It follows from the theorem that the triangulated monoidal structure
  on $H\udmod$ can also be described via this Verdier quotient
  structure. For instance, a short exact sequence of chain complexes
  of graded $H$-modules leads to a distinguished triangle in
  $\mc{D}^b(H,d)$, and thus passes to a distinguished triangle in
  $H\udmod$. Any distinguished triangle in $H\udmod$ is isomorphic to
  one arising in this manner.

  As an application, we can match the homological shifts on both sides
  of Rickard's equivalence. One starts with the short exact sequence
  of chain complexes
  \begin{equation}
    \begin{gathered}
      \xymatrix{
      & 0 \ar[r] & q^{2-2p} H \ar[r] \ar@{=}[d] &  q^{2-2p} U_{p-2} \ar[r] \ar@{=}[d] & 0 \\
      0\ar[r] & \Bbbk \ar@{=}[d] \ar[r] & q^{2-2p} H \ar[d] \ar[r] & q^{2-2p}U_{p-2} \ar[r] & 0\\
      0\ar[r] & \Bbbk \ar[r] & 0 & & }
    \end{gathered}
  \end{equation}
  The middle term is exact, and thus vanishes in the
  $\mc{D}^b(H,d)$. We have an isomorphism in $\mc{D}^b(H,d)$:
  \begin{equation}\label{eqn-shift1}
    \left(  0 \lra \underline{\Bbbk} \lra 0
    \right)  \cong   t_d\left(
      0\lra \underline{0} \lra q^{2-2p} H  \lra q^{2-2p} U_{p-2} \lra 0
    \right) 
\cong
    \left(
      0 \lra \underline{ q^{2-2p} H } \lra q^{2-2p} U_{p-2} \lra 0
    \right) ,
  \end{equation}
  where the underlined terms sit in homological degree zero. In the
  quotient category, we have the isomorphism
  \begin{equation}\label{eqn-shift2}
    \left(
      0 \lra \underline{ q^{2-2p} H } \lra q^{2-2p} U_{p-2} \lra 0
    \right)
    \cong
    \left(
      \underline{ 0 } \lra q^{2-2p} U_{p-2} \lra 0
    \right).
  \end{equation}
  It follows, upon shifting \eqref{eqn-shift2} by $t_d$, that
  \begin{equation}
    t_d\Bbbk \cong q^{2-2p}U_{p-2},
  \end{equation}
  and tensoring with $t_d\Bbbk$ gives the correct homological shift
  on $\mc{D}^b(H,d)/\mc{P}_H $ as described in \eqref{eqn-V-p-2}.
\end{rem}

Rickard's Theorem provides a way of taking slash cohomology of a chain
complex of graded $H$-modules. This is analogous to the usual
hyper-cohomology in the setting of $p$-complexes.

\begin{defn}\label{defhyperslash}
  If $M$ is a bounded complex of graded $H$-modules, we will denote
  its image in the stable category via Rickard's equivalence
  $\mc{D}^b(H,d)/\mc{P}_H$ by $\mH^\bullet_/(M)$.
\end{defn}

\subsubsection{Grothendieck ring}
As a matter of notation for later, we will regard $\Z[q,q^{-1}]$ as the Grothendieck ring of graded (chain complexes of) abelian groups. We will often use the usual \emph{quantum integers and binomial coefficients}
\begin{equation}
    [n]:=q^{n-1}+q^{n-3}+\cdots+q^{1-n}, \quad \quad \genfrac[]{0pt}{}{n}{k}:= \dfrac{[n][n-1]\cdots [n-k+1]}{[k][k-1]\cdots [1]} ,
\end{equation}
for any $n\in \N$ and $k\in \{0,\dots, n\}$.

The stable category $H\udmod$ is of particular interest to us because
it categorifies a ring closely related to the ring of cyclotomic integers at a $2p$th root of
unity.

\begin{lem}[{\cite[Proposition 5]{Hopforoots}}]
  The Grothendieck ring of the symmetric tensor triangulated category
  $H\udmod$ is isomorphic to
  \begin{equation}
    K_0(H\udmod) \cong \dfrac{\Z[q,q^{-1}]}{(1+q^2+\dots+q^{2(p-1)})}.
  \end{equation}
\end{lem}

For the rest of this paper, we denote
\begin{equation}
  \mathbb{O}_p:= \dfrac{\Z[q,q^{-1}]}{(1+q^2+\dots+q^{2(p-1)})}.
\end{equation}
We will refer to the map
\begin{equation}\label{eqn:Eulerchar}
K_0: H\udmod \lra \mathbb{O}_p,
\end{equation}
sending a graded $H$-module $M$ to the symbol
of its slash cohomology $\mH^\bullet_/(M)$ in $K_0(H\udmod)$ as the \emph{cyclotomic
Euler characteristic} of $M$.

Alternatively, via Rickard's Theorem (Theorem \ref{thm-Rickard}), the usual alternating
sum of graded dimensions, or \emph{graded Euler characteristic}, of
chain complexes identifies $K_0(\mc{D}^b(H,d))$ with $\Z[q,q^{-1}]$. Under
the graded Euler characteristic, the tensor ideal $\mc{P}_H$ is sent
to the principal ideal generated by $1+q^2+\dots+ q^{2(p-1)}$. These
identifications are readily seen to respect the monoidal structures
(again via the fact that $M\otimes H$ or $H\otimes M$ is graded free
for any $H$-module $M$). 
Given a bounded chain complex $M^\bullet$ of graded $H$-modules, we will refer to the composition map
\begin{equation}\label{eqn-cyclotomic-Euler-characteristic}
  \chi_{\mathbb{O}}:\mc{C}^b(H,d) \lra \mc{D}^b(H,d)\lra H\udmod \stackrel{K_0}{\lra} \mathbb{O}_p, \quad
  \chi_{\mathbb{O}}(M^\bullet)=\sum_i (-1)^i K_0(M^i)
\end{equation}
sending a complex of graded $H$-modules $M$ to the symbol
of its class in $K_0(H\udmod)$ as the \emph{cyclotomic
Euler characteristic} of $M$.

\subsubsection{Relative homotopy category}
We next extend the cyclotomic Euler characteristic to relative homotopy categories of $p$-DG algebras.

Suppose $(A,\dif_A)$ is a \emph{$p$-differential graded ($p$-DG)
algebra}, i.e., a graded algebra equipped with a differential $\dif_A$
of degree two, satisfying
\begin{equation}
  \dif_A^p(a)\equiv 0, \qquad \text{and} \qquad \dif_A(ab)=\dif_A(a)b+a\dif_A(b),
\end{equation}
for all $a,b\in A$. In other words, $A$ is an algebra object in the
graded module category of the Hopf algebra $H$, where the primitive
degree-two generator $\dif\in H$ acts on $A$ by the differential
$\dif_A$. 

Then, we may form the \emph{smash product algebra} $A\# H$  in this
case. As a $\Bbbk$-vector space, $A\# H$ is isomorphic to
$A\otimes H$; its multiplication is given by:
\begin{equation}
  (a\otimes \dif)(b\otimes \dif)=ab\otimes \dif^2+ a\dif_A(b)\otimes \dif.
\end{equation}
Thus, for a $p$-DG algebra $A$,  $A\otimes 1$ and $1\otimes H$ sit in
$A\# H$ as subalgebras by construction. 
We will often refer to modules in $A\#H\dmod$ as $H$-equivariant $A$-modules (or $p$-DG $A$-modules), and to their morphisms as $H$-equivariant (or $p$-DG) morphisms.

There is an exact forgetful functor between the usual homotopy
categories of chain complexes of graded $A\# H$-modules
\begin{equation}\label{eqn-forgetful-functor}
  \mc{F}: \mc{C}(A\# H)\lra \mc{C}(A).
\end{equation}
An object $K_\bullet$ in $\mc{C}(A\# H)$ lies inside the kernel of
the functor if and only if, when forgetting the $H$-module structure
on each term of $K_\bullet$, the complex of graded $A$-modules
$\mc{F}(K_\bullet)$ is null-homotopic. The null-homotopy map on
$\mc{F}(K_\bullet)$, though, is not required to intertwine
$H$-actions.

\begin{defn}\label{def-relative-homotopy-category}
  Given a $p$-DG algebra $(A,\dif_A)$, the \emph{relative homotopy
  category} is the Verdier quotient
  \[\mc{C}^\dif(A,d):=\dfrac{\mc{C}(A\#
    H,d)}{\mathrm{Ker}(\mc{F})}.\]
\end{defn}
The superscript $\dif$ in the definition is there to remind the reader of the
$H$-module structures on the objects.

The category $\mc{C}^\dif(A,d)$ is triangulated. By construction,
there is a factorization of the forgetful functor
\begin{gather}\label{eqn-localization-diagram}
  \NB{
    \tikz[xscale=3, yscale =0.8]{
      \node (AH) at (-1, 1) {$\mc{C}(A\# H,d)$};
      \node (A) at ( 1, 1) {$\mc{C}(A,d)$};
      \node (Adif) at ( 0, -1) {$\mc{C}^\dif(A,d)$};
      \draw[-to] (AH) -- (A) node[pos =0.5, above] {$\mc{F}$};
      \draw[-to] (AH) -- (Adif);
      \draw[-to] (Adif) -- (A);      
    }
  }  \ .
\end{gather}

Let us briefly discuss the triangulated structure of the relative homotopy category $\mc{C}^\dif(A,d)$. By construction, the shift functors $[\pm 1]$ are inherited from those of $\mc{C}(A\# H,d)$, which shift complexes one step to the left or right. 

For the usual homotopy category $\mc{C}(A,d) $ of an algebra, standard distinguished triangles arise from short exact sequences
  \[
 0 \lra M^\bullet \stackrel{f}{\lra} N^\bullet \stackrel{g}{\lra} L^\bullet \lra 0
  \]
of $A$-modules that are termwise split exact. The class of distinguished triangles in $\mc{C}(A,d)$ is declared to be those that are isomorphic to standard ones.
For distinguished triangles in the relative homotopy category, similarly, we have the following construction.

\begin{lem}\label{lem-construction-of-triangle} \cite[Lemma 2.3]{QRSW1}
  A short exact sequence of chain complexes of $A\#H$-modules
  \[
 0 \lra M^\bullet \stackrel{f}{\lra} N^\bullet \stackrel{g}{\lra} L^\bullet \lra 0
  \]
 that is termwise $A$-split exact gives rise to a distinguished triangle in $\mc{C}^\dif(A,d)$. Conversely, any distinguished triangle in $\mc{C}^\dif(A,d)$ is isomorphic to one that arises in this form.
\end{lem}

Since the objects of $\mathrm{Ker}(\mc{F})$ are clearly contained in the class of acyclic objects in $\mc{C}(A\#H,d)$, one has a well-defined functor
\begin{equation}
    \mc{C}^{\dif}(A,d)\to \mc{D}(A\#H,d) .
\end{equation}
It follows that the cyclotomic Euler characteristic $\chi_{\mathbb{O}}$ \eqref{eqn-cyclotomic-Euler-characteristic} is well-defined on bounded relative homotopy classes. In the special case when $A=\Bbbk$, the kernel of the forgetful functor $\mc{F}$ in diagram \eqref{eqn-localization-diagram} consists precisely of acyclic complexes. Therefore, we have $\mc{C}^\dif(\Bbbk,d)=\mc{D}(H,d)$, and $\chi_{\mathbb{O}}$ is thus equal to the composition map (cf.~equation \eqref{eqn-cyclotomic-Euler-characteristic})
\begin{equation}\label{eqn-relative-homotopy-equals-derived}
    \chi_{\mathbb{O}}:\mc{D}(H,d)\lra H\udmod\stackrel{K_0}{\lra} \mathbb{O}_p.
\end{equation}

\subsection{Graded representation theory of \texorpdfstring{$\mathfrak{sl}(2)$}{sl(2)} in prime characteristic} \label{sec:sl(2)-rep-theory}

In this section, we collect some facts about the distribution algebra of the group $G=\mathrm{SL}(2)$ in positive characteristic. We refer to \cite[Part I, Chapter 7]{Jantzen} for a more general approach. 

Consider the special linear group $\mathrm{SL}(2)$ over a field $\Bbbk$ of characteristic $p$, and let $\mathfrak{sl}(2)=\Bbbk\langle e,f,h\rangle$ be the Lie algebra of $\mathrm{SL}(2)$ over $\Bbbk$.  Denote by
$\Ul_p(\mathfrak{sl}(2))$ the distribution algebra of $\mathrm{SL}(2)$. It is the cocommutative Hopf algebra generated by the divided powers
\[
 e^{(r)},\qquad f^{(r)},\qquad \binom{h}{r},
 \qquad r\geq 0,
\]
subject to the relations
\begin{align} 
e^{(r)}e^{(s)} =\binom{r+s}{r}e^{(r+s)},\quad \quad f^{(r)}f^{(s)} =\binom{r+s}{r}f^{(r+s)},\\
\binom{h}{m}e^{(r)} =e^{(r)}\binom{h+2r}{m},\quad \quad \binom{h}{m}f^{(r)} =f^{(r)}\binom{h-2r}{m},\\ e^{(r)}f^{(s)} =\sum_{j=0}^{\min(r,s)} f^{(s-j)} \binom{h-r-s+2j}{j} e^{(r-j)},
\end{align}
where 
\[
\binom{h+a}{m}
=
\sum_{i=0}^{m}
\binom{a}{i}\binom{h}{m-i} .
\]

The first Frobenius kernel $\ul_p(\mathfrak{sl}(2))$ is the subalgebra of the distribution algebra $\Ul_p(\mathfrak{sl}(2))$ generated by $e$, $f$ and $h$, which satisfy the relations
 \[e^p =f^p=0,\quad \quad h^p-h=0.\]
Both $\Ul_p(\mathfrak{sl}(2))$ and $\ul_p(\mathfrak{sl}(2))$ are graded algebras by setting
\[
 \deg(e^{(r)})=2r,\qquad
 \deg(f^{(r)})=-2r,\qquad
 \deg\binom{h}{r}=0.
\]
We regard both as $p$-DG algebras by equipping them with the zero $p$-differential. When no confusion can be caused, we abbreviate $\Ul_p(\mathfrak{sl}(2))$ and $\ul_p(\mathfrak{sl}(2))$ as $\Ul_p$ and $\ul_p$ respectively.

The group $\mathrm{SL}(2)$ acts naturally on the $2$-dimensional vector space $T(1)=V_1=\Bbbk\langle v_1, v_{-1}\rangle$. Upon symmetrization, we obtain the \emph{Weyl module} for any $n\in \N$:
\begin{equation}
    W(n):=\mathrm{Sym}^n T(1).
\end{equation}
Define the \emph{dual Weyl module} $M(n):=W(n)^*$, equipped with the twisted $\mathrm{SL}(2)$ action coming from the Cartan involution. 

\begin{defn}\label{def-tilting}
A finite-dimensional
$\mathrm{SL}(2)$-module is called tilting if it admits both a Weyl filtration and a
dual Weyl filtration. For every $n\geq 0$, there is a unique
indecomposable tilting module $T(n)$ of highest weight $n$.
\end{defn}

Let $\operatorname{Tilt}$ be the additive monoidal category generated
by the modules $T(n)$. The natural 2-dimensional module $T(1)$
generates this category under tensor products, direct sums and direct
summands. Thus every $T(n)$ occurs as a direct summand of
$V_1^{\otimes n}$.

For $0\leq n\leq p-1$, all three highest-weight modules coincide:
\[
 T(n)\cong W(n)\cong M(n).
\]
In particular, $T(p-1)$ is the \emph{Steinberg module}.

\subsubsection{Clebsch--Gordan rule for tilting modules}

We refer the reader to \cite{BLS, SuttonTubbenhauerWedrichZhu} for the following fusion rule of tilting modules. For any positive integer $n$, write its $p$-adic expansion as $n=\sum_{i=0}^{j}a_i p^i$, with $0\leq a_i\leq p-1$, $  a_j\neq0$, and denote by
\[
  n=\sum_{i=0}^{j}a_i p^i=[a_j,\ldots,a_0]_p.
\]

Define the tail length by
\[
  \tl(n)
  :=
  \begin{cases}
  j+1 & \text{if } a_j=\cdots=a_0=p-1,\\
  \min\{\,i\geq0\mid a_i\neq p-1\,\}
      & \text{otherwise}.
  \end{cases}
\]
Thus
\[
  a_0=\cdots=a_{\tl(n)-1}=p-1,
  \qquad
  a_{\tl(n)}\neq p-1,
\]
where, in the case $\tl(n)=j+1$, we set $a_{j+1}=0$.

For $0\leq i\leq \tl(n)$, define
\[
  x_i :=
  \begin{cases}
    0, & a_i=0,\\
    0, & i=j \text{ and } a_j=1,\\
    2, & a_i=1 \text{ and } i<j,\\
    1, & a_i>1.
  \end{cases}
\]
Then
\[
  T(n-1)\otimes T(1)
  \cong
  T(n)
  \oplus
  \bigoplus_{i=0}^{\tl(n)}
  T(n-2p^i)^{\oplus x_i}.
\]

\subsection{Diagrammatics of tilting modules} \label{sec:TLidemp}
Define
\[
  \NB{\tikz[scale = 0.6]{\begin{scope}
  \coordinate (lb) at ( 0, -1);
  \coordinate (lt) at ( 0, 1);
   \coordinate (rb) at ( 1, -1);
      \coordinate (rt) at ( 1, 1);
  \draw[] (lb) -- (rt) node[above, pos =1] {} node[below, pos =0] {};
    \draw[] (lt) -- (rb) node[above, pos =1] {} node[below, pos =0] {};
\end{scope}}}
 :=
    \NB{\tikz[scale = 0.6]{\begin{scope}
  \coordinate (bl) at (-0.5, -1);
  \coordinate (br) at ( 0.5, -1);
  \coordinate (tl) at (-0.5,  1);
  \coordinate (tr) at ( 0.5,  1);
    \coordinate (ml) at (-0.5,  -.8);
        \coordinate (Ml) at (-0.5,  .8);
 \coordinate (mr) at (0.5,  -.6);
\coordinate (Mr) at (0.5,  .6);

   \draw[] (bl) -- (tl);
  
    \draw[] (br) -- (tr);
  

\end{scope}}}
    -
   \NB{\tikz[scale = 0.6]{\begin{scope}
  \coordinate (bl) at (0, -1);
  \coordinate (br) at (2, -1);
  \coordinate (tl) at (0,  1);
  \coordinate (tr) at (2,  1);
    \coordinate (ml) at (-0.5,  -.8);
        \coordinate (Ml) at (-0.5,  .8);
 \coordinate (mr) at (0.5,  -.6);
\coordinate (Mr) at (0.5,  .6);

   \draw [] (.5, 1) arc (-180:0:0.5) ;
   \draw (.5, -1) arc (180:0:0.5) coordinate[pos = 0.2] (X) coordinate[pos = 0.8] (Y);
\end{scope}}} 
\]
where
\[
  \begin{array}{crcl}
    \NB{\tikz[scale = 0.6]{\begin{scope}
  \coordinate (bl) at (0, -1);
  \coordinate (br) at (2, -1);
  \coordinate (tl) at (0,  1);
  \coordinate (tr) at (2,  1);
    \coordinate (ml) at (-0.5,  -.8);
        \coordinate (Ml) at (-0.5,  .8);
 \coordinate (mr) at (0.5,  -.6);
\coordinate (Mr) at (0.5,  .6);

   \draw [] (.5, 1) arc (-180:0:0.5) ;

\end{scope}}}:& V_0 & \rightarrow& V_1^{\otimes 2} \\
    &1 &\mapsto &v_1 \otimes v_{-1} - v_{-1} \otimes v_1   
  \end{array}
  \qquad\text{and}\qquad 
  \begin{array}{crcl}
  \NB{\tikz[scale = 0.6]{\begin{scope}
  \coordinate (bl) at (0, -1);
  \coordinate (br) at (2, -1);
  \coordinate (tl) at (0,  1);
  \coordinate (tr) at (2,  1);
    \coordinate (ml) at (-0.5,  -.8);
        \coordinate (Ml) at (-0.5,  .8);
 \coordinate (mr) at (0.5,  -.6);
\coordinate (Mr) at (0.5,  .6);

   \draw (.5, -1) arc (180:0:0.5) coordinate[pos = 0.2] (X) coordinate[pos = 0.8] (Y);

\end{scope}}} \colon& V_1^{\otimes 2}& \rightarrow &V_0 \\
&v_i \otimes v_i &\mapsto& 0\\
&v_1 \otimes v_{-1} &\mapsto& 1 \\
&v_{-1} \otimes v_1& \mapsto& -1\\
\end{array}.
                           \]
More generally, define, for $i=1,\ldots, n-1$, an endomorphism
\[
s_i =: 
\NB{\tikz[scale = 0.6]{\begin{scope}
 \coordinate (llb) at ( -2, -1); 
  \coordinate (llt) at ( -2, 1);
   \coordinate (rrb) at ( 3, -1);
  \coordinate (rrt) at ( 3, 1);
  \coordinate (lb) at ( 0, -1);
  \coordinate (lt) at ( 0, 1);
   \coordinate (rb) at ( 1, -1);
      \coordinate (rt) at (1, 1);
  \draw[] (lb) -- (rt) node[above, pos =1] {} node[below, pos =0] {};
    \draw[] (lt) -- (rb) node[above, pos =1] {} node[below, pos =0] {};
      \draw[] (llb) -- (llt) node[right, pos =0.5] {} node[below, pos =0] {};
          \draw[] (rrb) -- (rrt) node[left, pos =.5] {}  node[below, pos =0] {};
          \node at (-1, 0) {$\dots$};
           \node at (2, 0) {$\dots$};
\end{scope}}} 
\]
giving rise to an action of the symmetric group $S_n$ on $V_1^{\otimes n}$.
When $n\leq p-1$, this leads to an endomorphism $\jw_n \colon V_1^{\otimes n} \rightarrow V_1^{\otimes n}$ defined by
\begin{equation} \label{eq:defjwn}
\jw_n = \frac{1}{n!} \sum_{w \in S_n} w \ .
\end{equation}
Graphically, we depict this morphism by:
\[ 
\jw_n =
  \NB{\tikz[scale = 0.6]{\begin{scope}

  \draw (0,0) rectangle (2,1) node[pos=.5] {$\jw_n$};
      \draw[-] (.25,1) -- (.25,1.5) node[right, pos =.5] {}  node[above, pos =1] {} node[below, pos =0] {};
       \node at (1, 1.25) {$\dots$};
    \draw[-] (1.75,1) -- (1.75,1.5) node[above, pos =1] {} node[below, pos =0] {};
          \draw[-] (.25,-0.5) -- (.25,0) node[right, pos =.5] {}  node[above, pos =1] {} node[below, pos =0] {};
            \node at (1, -0.25) {$\dots$};
    \draw[-] (1.75,-0.5) -- (1.75,0) node[above, pos =1] {} node[below, pos =0] {};
\end{scope}}} \ .
\]
They satisfy the recursive relation 
\begin{equation} \label{eqn:Vrec}
\jw_0=1, \quad  \quad  \jw_1=\NB{\tikz[scale=0.5]{\begin{scope}
  \coordinate (bm) at ( 0, -2);
  \coordinate (tm) at ( 0, 2);
  \draw (bm) -- (tm) ;
\end{scope}}}~,\quad \textrm{and}  \quad
   \NB{\tikz[scale=0.5]{
\begin{scope}
  \draw (0, 0) -- + ( 0, 4);
    \draw (1.5, 0) -- + ( 0, 4);
  \draw (2, 0) -- + ( 0, 4);
  \node at (0.75, 0.5) {$\dots$};
  \node at (0.75, 3.5) {$\dots$};
    \filldraw[fill=white, draw = black] (-0.2, 1.5) rectangle (2.2, 2.5) coordinate[pos=0.5] (B);
    \node at (B) {$\jw_{n+1}$};
\end{scope}}} = \NB{\tikz[scale=0.5]{
\begin{scope}
  \draw (0, 0) -- + ( 0, 4);
    \draw (1.5, 0) -- + ( 0, 4);
  \draw (2, 0) -- + ( 0, 4);
  \node at (0.75, 0.5) {$\dots$};
  \node at (0.75, 3.5) {$\dots$};
    \filldraw[fill=white, draw = black] (-0.2, 1.5) rectangle (1.7, 2.5) coordinate[pos=0.5] (B);
    \node at (B) {$\jw_{n}$};
\end{scope}}} - \frac{n}{n+1}   ~\NB{\tikz[scale=0.5]{
\begin{scope}
  \draw (0, 0) -- + ( 0,4);
  \draw (1.5, 1.5) -- + (0,2);
  \node at (1, 0) {$\dots$};
  \node at (1, 4) {$\dots$};
  \node at (0.75, 2) {$\dots$};
  \draw (2.5,0) -- +(0, 1.6) arc (0:180:0.25) -- +(0, -1.6);
  \draw (2.5,4) -- +(0, -1.6) arc (0:-180:0.25) -- +(0, 1.6);
  \filldraw[fill=white, draw = black] (-0.1, 0.3) rectangle (2.1, 1.5) coordinate[pos=0.5] (A);
    \filldraw[fill=white, draw = black] (-0.1, 3.7) rectangle (2.1, 2.5) coordinate[pos=0.5] (B);
    \node at (A) {$\jw_{n}$};
    \node at (B) {$\jw_{n}$};
\end{scope}}} \ .
\end{equation}
Due to the denominators appearing in the formula, $\jw_n$ is well-defined only for $n\in \{0,\dots, p-1\}$.

When $n=p-1$, the Jones--Wenzl projector $\jw_{p-1}$ is an indecomposable idempotent, and so is $\jw_{p-1} \otimes \Id$:
\begin{equation}\label{eqn:Pp-1Pp}
  \NB{\tikz[scale = 0.6]{\begin{scope}

  \draw (0,0) rectangle (2,1) node[pos=.5] {$\jw_{p-1}$};
      \draw[-] (.25,1) -- (.25,1.5) node[right, pos =.5] {}  node[above, pos =1] {} node[below, pos =0] {};
       \node at (1, 1.25) {$\dots$};
    \draw[-] (1.75,1) -- (1.75,1.5) node[above, pos =1] {} node[below, pos =0] {};
          \draw[-] (.25,-0.5) -- (.25,0) node[right, pos =.5] {}  node[above, pos =1] {} node[below, pos =0] {};
            \node at (1, -0.25) {$\dots$};
    \draw[-] (1.75,-0.5) -- (1.75,0) node[above, pos =1] {} node[below, pos =0] {};
\end{scope}}}~,\quad \quad
    \NB{\tikz[scale = 0.6]{\begin{scope}

  \draw (0,0) rectangle (2,1) node[pos=.5] {$\jw_{p-1}$};
      \draw[-] (.25,1) -- (.25,1.5) node[right, pos =.5] {}  node[above, pos =1] {} node[below, pos =0] {};
       \node at (1, 1.25) {$\dots$};
    \draw[-] (1.75,1) -- (1.75,1.5) node[above, pos =1] {} node[below, pos =0] {};
          \draw[-] (.25,-0.5) -- (.25,0) node[right, pos =.5] {}  node[above, pos =1] {} node[below, pos =0] {};
            \node at (1, -0.25) {$\dots$};
    \draw[-] (1.75,-0.5) -- (1.75,0) node[above, pos =1] {} node[below, pos =0] {};
\draw[-] (2.25,-0.5) -- (2.25,1.5) node[above, pos =1] {} node[below, pos =0] {};
\end{scope}}}~.
\end{equation}

Write $V=T(1)$ for the natural two-dimensional $SL(2)$-module.

With the conventions of \cite{SuttonTubbenhauerWedrichZhu}, a
closed circle evaluates to $-[2]_{\xi}$; we specialize $\xi=-1$. Since
\(
  [r]_{-1}=(-1)^{r-1}r,
\)
we obtain
\(
  \bigcirc=-[2]_{-1}=2,
\)
which agrees with the pivotal convention used above.

Let $\TL_{\Bbbk,-1}$ denote the resulting Temperley--Lieb
category.  After additive idempotent completion, the pivotal
functor 
\[
  \TL_{\Bbbk,-1}
  \longrightarrow
  \Ul_p 
\qquad
  n\longmapsto V^{\otimes n},
\]
identifies the Karoubi envelope of $\TL_{\Bbbk,-1}$ with the
category of tilting $G$-modules.  For every $m\geq 0$, let  
$  E_m\in\End_{\TL_{\Bbbk,-1}}(m)$
denote the mixed Jones--Wenzl idempotent whose image is the
indecomposable tilting module $T(m)$.
For $0\leq m\leq p-1$, this is the ordinary Jones--Wenzl
projector:  $E_m=\mathsf V_m$.

Graphically, we depict\footnote{Following the idea of \cite{SuttonTubbenhauerWedrichZhu}, the diagram for $E_m$ emphasizes the lack of left-right symmetry of these morphisms (see for instance \eqref{eqn:Pp-1Pp}).} this morphism by:
\[ 
E_m =
  \NB{\tikz[scale = 0.6]{\begin{scope}
    \draw (-1, -1) -- + (0,2);
  \draw (1.25, -1) -- + (0,2);
  \node at (0, -0.85) {$\dots$};
  \node at (0, 0.85) {$\dots$};
    \tikztilting{.00}{0}{1}{1}{$\scriptstyle{m}$}
\end{scope}
}} \ .
\]

In prime characteristic, these idempotents recover the (p)-Jones--Wenzl projectors introduced by Burrull--Libedinsky--Sentinelli \cite{BLS}. The recursive description used below, including the explicit idempotents projecting onto the lower tilting summands, is obtained by specializing the mixed-characteristic construction of Sutton--Tubbenhauer--Wedrich--Zhu \cite{SuttonTubbenhauerWedrichZhu}.

\subsubsection{The lower-summand idempotents}

Fix $0\leq i\leq\tl(n)$ and put $w=p^i$. Let
$X_n^i\in\End(E_{n-1}\otimes\id_1)$ be defined by
\begin{equation}
    X_n^0 = \NB{\tikz[scale = 0.5]{\begin{scope}
  \draw (-1, -0.5) -- + (0,7);
  \draw (1.25, -0.5) -- + (0,7);
    \draw (2.25, -0.5) -- + (0,7);
  \node at (0, 0) {$\dots$};
  \node at (0, 6) {$\dots$};
  \node at (0, 2) {$\dots$};
  \node at (0, 4) {$\dots$};
  \draw (3.25,-0.5) -- +(0, 2.75) arc (0:180:0.25) -- +(0, -2.75);
  \draw (3.25,6.5) -- +(0, -2.75) arc (0:-180:0.25) -- +(0, 2.75);
  \node at (1.75, 0) {$\dots$};
    \node at (1.75, 6) {$\dots$};
      \node at (1.75, 2) {$\dots$};
  \node at (1.75, 4) {$\dots$};
  \tikztilting{.75}{0.9}{1.52}{1}{$\scriptstyle{n-1}$}
  \tikztilting{.75}{5.1}{1.52}{1}{$\scriptstyle{n-1}$}
  \tikztilting{.00}{2.9}{1}{1}{$\scriptstyle{n-2a_0+1}$}

\end{scope}
}}, \qquad \text{and} \qquad
    X_n^i = \NB{\tikz[scale = 0.5]{\begin{scope}
  \draw (-1, -0.5) -- + (0,7);
  \draw (1.25, -0.5) -- + (0,7);
    \draw (2.25, -0.5) -- + (0,7);
  \node at (0, 0) {$\dots$};
  \node at (0, 6) {$\dots$};
  \node at (0, 2) {$\dots$};
  \node at (0, 4) {$\dots$};
  \draw (3.25,-0.5) -- +(0, 2.75) arc (0:180:0.25) -- +(0, -2.75);
  \draw (3.25,6.5) -- +(0, -2.75) arc (0:-180:0.25) -- +(0, 2.75);
  \node at (1.75, 0) {$\dots$};
    \node at (1.75, 6) {$\dots$};
      \node at (1.75, 2) {$\dots$};
  \node at (1.75, 4) {$\dots$};
  \tikztilting{.75}{0.9}{1.52}{1}{$\scriptstyle{n-w}$}
  \tikztilting{.75}{5.1}{1.52}{1}{$\scriptstyle{n-w}$}
  \tikztilting{.00}{2.9}{1.2}{1}{$\scriptstyle{n-(2a_i-1)w}$}
\end{scope}
}}.
\end{equation}
Let $Y_n^i\in\End(E_{n-1}\otimes\id_1)$ be defined by
\begin{equation}
    Y_n^0 = \NB{\tikz[scale = 0.5]{\begin{scope}
  \draw (-2, -0.5) -- + (0,7);
  \draw (.25, -0.5) -- + (0,7);
    \draw (1.25, -0.5) -- + (0,7);
  \node at (-1, 0) {$\dots$};
  \node at (-1, 6) {$\dots$};
  \node at (-1, 2) {$\dots$};
  \node at (-1, 4) {$\dots$};
    \node at (2.2, 2.5) {$\scriptstyle a_0-1$};
        \node at (2.2, 3.45) {$\scriptstyle a_0-1$};
  \draw (2.25,.5) -- +(0, 1.25) arc (0:180:0.25) -- +(0, -1.25);
  \draw (2.25,5.5) -- +(0, -1.25) arc (0:-180:0.25) -- +(0, 1.25);
    \draw (4,-0.5) -- +(0, 2.25) arc (0:180:0.25) -- +(0, -2.25);
  \draw (4,6.5) -- +(0, -2.25) arc (0:-180:0.25) -- +(0, 2.25);
  \node at (.75, 0) {$\dots$};
    \node at (.75, 6) {$\dots$};
      \node at (.75, 2) {$\dots$};
  \node at (.75, 4) {$\dots$};
  \tikztilting{.75}{0.9}{2}{1}{$\scriptstyle{n-1}$}
  \tikztilting{.75}{5.1}{2}{1}{$\scriptstyle{n-1}$}
  \tikztilting{-.60}{2.9}{1}{1}{$\scriptstyle{n-2a_0+1}$}

\end{scope}
}}, \qquad \text{and} \qquad
    Y_n^i = \NB{\tikz[scale = 0.5]{\begin{scope}
  \draw (-2, -0.5) -- + (0,7);
  \draw (.25, -0.5) -- + (0,7);
    \draw (1.25, -0.5) -- + (0,7);
  \node at (-1, 0) {$\dots$};
  \node at (-1, 6) {$\dots$};
  \node at (-1, 2) {$\dots$};
  \node at (-1, 4) {$\dots$};
    \node at (2.2, 2.5) {$\scriptstyle a_i w$};
        \node at (2.2, 3.45) {$\scriptstyle a_i w$};
          \node at (3.75, 2.5) {$\scriptstyle  w$};
        \node at (3.75, 3.45) {$\scriptstyle  w$};
  \draw (2.25,.5) -- +(0, 1.25) arc (0:180:0.25) -- +(0, -1.25);
  \draw (2.25,5.5) -- +(0, -1.25) arc (0:-180:0.25) -- +(0, 1.25);
    \draw (4,-0.5) -- +(0, 2.25) arc (0:180:0.25) -- +(0, -2.25);
  \draw (4,6.5) -- +(0, -2.25) arc (0:-180:0.25) -- +(0, 2.25);
  \node at (.75, 0) {$\dots$};
    \node at (.75, 6) {$\dots$};
      \node at (.75, 2) {$\dots$};
  \node at (.75, 4) {$\dots$};
  \tikztilting{.75}{0.9}{2}{1}{$\scriptstyle{n-w}$}
  \tikztilting{.75}{5.1}{2}{1}{$\scriptstyle{n-w}$}
  \tikztilting{-.60}{2.9}{1.2}{1}{$\scriptstyle{n-(2a_i-1)w}$}

\end{scope}
}} \qquad \text{if $i>0$}.
\end{equation}
The diagrams $X_n^i$ and $Y_n^i$ are obtained from the diagrams in
\cite[Definition~4.14]{SuttonTubbenhauerWedrichZhu} after specializing
$\ell=p$, $\xi=-1$ and replacing $v$ by $n$.

Suppose first that $a_i>1$ and define
\[
  H_n^i:=(a_i-1)Y_n^i
\]
and
\[
  \epsilon_{n,i}:=
  \begin{cases}
    1,
      &i<j\text{ and }2\leq a_i\leq p-2,\\
    0,
      &i=j\text{ or }a_i=p-1.
  \end{cases}
\]
Then the primitive idempotent projecting onto the summand
$T(n-2p^i)$ is
\begin{equation}
  \label{eq:tilting-B-QRSW-normalization}
  B_n^i
  =
  \frac{a_i-1}{a_i}\,X_n^i
  +
  \epsilon_{n,i}\frac{2}{a_i^2}\,H_n^i.
\end{equation}

Suppose next that $a_i=1$. If $i=j$, we set
\[
  A_n^j=\widehat A_n^j=0.
\]
Otherwise, when $i<j$, define
\begin{equation}
    A_n^0 = \NB{\tikz[scale = 0.5]{\begin{scope}
  \draw (0, -0.5) -- + (0,7);
  \draw (1.5, -0.5) -- + (0,7);
  \node at (.75, -0.25) {$\dots$};
  \node at (.75, 6.25) {$\dots$};
  \node at (.75, 2) {$\dots$};
  \node at (.75, 4) {$\dots$};
   \draw (2.75, -0.5) -- + (0,4);
  \draw (2.75,6) -- +(0, -1.5) arc (0:-180:0.25) -- +(0, 1.5);
  \draw (2.75,3.5) -- +(0, 0.25) arc (180:90:0.5);
    \draw (3.75,6.5) -- +(0, -1.75) arc (0:-90:0.5);
     \draw (3.5,-0.5) -- +(0, 2) arc (0:180:0.25) -- +(0, -2);
    \tikztilting{1.45}{0.5}{1.1}{1}{$\scriptstyle n-1$}
    \tikztilting{1.45}{5.5}{1.1}{1}{$\scriptstyle n-1$}
    \tikztilting{1.45}{3}{1.1}{1}{$\scriptstyle n-2$}
\end{scope}
}}
+ \quad \NB{\tikz[scale = 0.5]{
\begin{scope}
  \draw (0, -0.5) -- + (0,7);
  \draw (1.5, -0.5) -- + (0,7);
  \node at (.75, -0.25) {$\dots$};
  \node at (.75, 6.25) {$\dots$};
  \node at (.75, 2) {$\dots$};
  \node at (.75, 4) {$\dots$};
  \draw (2.75,0) -- +(0, 1.5) arc (0:180:0.25) -- +(0, -.5);
  \draw (2.75,6) -- +(0, -1.5) arc (0:-180:0.25) -- +(0, 1.5);
  \draw (2.75,3.5) -- +(0, 0.25) arc (180:90:0.5);
    \draw (3.75,6.5) -- +(0, -1.75) arc (0:-90:0.5);
      \draw (2.75,2.5) -- +(0, -0.25) arc (180:270:0.5);
    \draw (3.75,-0.5) -- +(0, 1.75) arc (0:90:0.5);
    \tikztilting{1.45}{0.5}{1.1}{1}{$\scriptstyle n-1$}
    \tikztilting{1.45}{5.5}{1.1}{1}{$\scriptstyle n-1$}
    \tikztilting{1.45}{3}{1.1}{1}{$\scriptstyle n-2$}
\end{scope}
}}
\quad \text{and}\quad
    A_n^i = \NB{\tikz[scale = 0.5]{\begin{scope}
  \draw (0, -0.5) -- + (0,7);
  \draw (1.5, -0.5) -- + (0,7);
  \node at (.75, -0.25) {$\dots$};
  \node at (.75, 6.25) {$\dots$};
  \node at (.75, 2) {$\dots$};
  \node at (.75, 4) {$\dots$};
  
  \node at (3.5, 4) {$\scriptstyle  w$};
    \node at (3.5, 2) {$\scriptstyle  w$};
      \node at (2.5, 4) {$\scriptstyle  w$};

   \draw (2.75, -0.5) -- + (0,4);
  \draw (2.75,6) -- +(0, -1.5) arc (0:-180:0.25) -- +(0, 1.5);
  \draw (2.75,3.5) -- +(0, 0.25) arc (180:90:0.5);
    \draw (3.75,6.5) -- +(0, -1.75) arc (0:-90:0.5);
     \draw (3.5,-0.5) -- +(0, 2) arc (0:180:0.25) -- +(0, -2);
    \tikztilting{1.45}{0.5}{1.1}{1}{$\scriptstyle n-w$}
    \tikztilting{1.45}{5.5}{1.1}{1}{$\scriptstyle n-w$}
    \tikztilting{1.45}{3}{1.1}{1}{$\scriptstyle n-2w$}
\end{scope}
}}
+ \quad \NB{\tikz[scale = 0.5]{\begin{scope}
  \draw (0, -0.5) -- + (0,7);
  \draw (1.5, -0.5) -- + (0,7);
  \node at (.75, -0.25) {$\dots$};
  \node at (.75, 6.25) {$\dots$};
  \node at (.75, 2) {$\dots$};
  \node at (.75, 4) {$\dots$};
\node at (3.5, 4) {$\scriptstyle  w$};
\node at (3.5, 2) {$\scriptstyle  w$};
\node at (2.5, 4) {$\scriptstyle  w$};
\node at (2.5, 2) {$\scriptstyle  w$};
\draw (2.75,0) -- +(0, 1.5) arc (0:180:0.25) -- +(0, -.5);
  \draw (2.75,6) -- +(0, -1.5) arc (0:-180:0.25) -- +(0, 1.5);
  \draw (2.75,3.5) -- +(0, 0.25) arc (180:90:0.5);
    \draw (3.75,6.5) -- +(0, -1.75) arc (0:-90:0.5);
      \draw (2.75,2.5) -- +(0, -0.25) arc (180:270:0.5);
    \draw (3.75,-0.5) -- +(0, 1.75) arc (0:90:0.5);
    \tikztilting{1.45}{0.5}{1.1}{1}{$\scriptstyle n-w$}
    \tikztilting{1.45}{5.5}{1.1}{1}{$\scriptstyle n-w$}
    \tikztilting{1.45}{3}{1.1}{1}{$\scriptstyle n-2w$}
\end{scope}
}}
\quad \text{ for } i>0.
\end{equation}
Let $(A_n^i)^\ell$ denote the reflection of $A_n^i$ in a
horizontal line. The diagrammatic computations in the proof of
\cite[Theorem~4.15]{SuttonTubbenhauerWedrichZhu} give
\begin{equation}
  (A_n^i)^2=A_n^i,
  \qquad
  \bigl((A_n^i)^\ell\bigr)^2=(A_n^i)^\ell,
  \qquad
  (A_n^i)^\ell A_n^i=0.
\end{equation}
The opposite product $A_n^i(A_n^i)^\ell$ need not vanish. We therefore
orthogonalize the second idempotent and define
\begin{equation}
  \label{eq:tilting-Ahat}
  \widehat A_n^i
  :=
  (A_n^i)^\ell
  -
  A_n^i(A_n^i)^\ell.
\end{equation}
Then
\begin{equation}
  (A_n^i)^2=A_n^i,
  \qquad
  (\widehat A_n^i)^2=\widehat A_n^i,
  \qquad
  A_n^i\widehat A_n^i
  =
  \widehat A_n^iA_n^i
  =0.
\end{equation}
Thus $A_n^i$ and $\widehat A_n^i$ are mutually orthogonal primitive
idempotents, and both have image isomorphic to $T(n-2p^i)$.

Finally, set
\[
  T_n^i:=
  \begin{cases}
    0,
      &a_i=0,\\[2pt]
    A_n^i+\widehat A_n^i,
      &a_i=1,\\[2pt]
    B_n^i,
      &a_i>1.
  \end{cases}
\]

\begin{thm}[Tilting Jones--Wenzl recursion]
  \label{thm:tilting-JW-recursion}
  Let $n=[a_j,\ldots,a_0]_p$. There is a decomposition
  \begin{equation}
    \label{eq:tilting-fusion-idempotents}
    E_{n-1}\otimes\id_1
    =
    E_n
    +
    \sum_{i=0}^{\tl(n)}T_n^i
  \end{equation}
  into mutually orthogonal idempotents.

  For $a_i=1$, the idempotent
  $T_n^i=A_n^i+\widehat A_n^i$ is the sum of two mutually orthogonal
  primitive idempotents; for $a_i>1$, $T_n^i=B_n^i$ is primitive.
  Consequently,
  \begin{equation}
    \label{eq:tilting-JW-inductive}
    E_n
    =
    E_{n-1}\otimes\id_1
    -
    \sum_{i=0}^{\tl(n)}T_n^i.
  \end{equation}
  Under the Temperley--Lieb--tilting equivalence,
  \[
    \operatorname{im}(E_n)\cong T(n),
    \qquad
    \operatorname{im}(T_n^i)
    \cong
    T(n-2p^i)^{\oplus x_i},
  \]
  where
  \[
    x_i=
    \begin{cases}
      0,
        &a_i=0
          \text{ or }(i=j\text{ and }a_j=1),\\
      2,
        &a_i=1\text{ and }i<j,\\
      1,
        &a_i>1.
    \end{cases}
  \]
  Equivalently,
  \[
    T(n-1)\otimes T(1)
    \cong
    T(n)
    \oplus
    \bigoplus_{i=0}^{\tl(n)}
       T(n-2p^i)^{\oplus x_i}.
  \]
\end{thm}

\begin{proof}
The formulas for the lower summands are obtained by specializing the
construction of \cite[Theorems~3.19 and~4.15]{SuttonTubbenhauerWedrichZhu}
to mixed characteristic $(p,p)$ and Temperley--Lieb parameter $\xi=-1$.
For $a_i>1$, the primitive idempotents $B_n^i$ are unchanged.

Suppose now that $a_i=1$ and $i<j$. For brevity, write
\[
  A:=A_n^i,
  \qquad
  F:=(A_n^i)^\ell.
\]
The diagrammatic computations in the proof of
\cite[Theorem~4.15]{SuttonTubbenhauerWedrichZhu} give
\[
  A^2=A,
  \qquad
  F^2=F,
  \qquad
  FA=0.
\]
The assertion that the opposite product $AF$ also vanishes does not
follow from the displayed multiplication relations. Instead, define
\[
  \widehat A:=F-AF.
\]
Then
\[
\begin{aligned}
  \widehat A^2
  &=(F-AF)^2\\
  &=F^2-FAF-AF^2+AFAF\\
  &=F-AF\\
  &=\widehat A,
\end{aligned}
\]
where we used $F^2=F$ and $FA=0$. Moreover,
\[
  A\widehat A
  =AF-A^2F
  =0,
  \qquad
  \widehat A A
  =FA-AFA
  =0.
\]
Thus $A_n^i$ and $\widehat A_n^i$ are mutually orthogonal idempotents.

The idempotents $F=(A_n^i)^\ell$ and $\widehat A_n^i$ define isomorphic
objects in the Karoubi envelope. Indeed,
\[
  \widehat A_n^i (A_n^i)^\ell=\widehat A_n^i,
  \qquad
  (A_n^i)^\ell\widehat A_n^i=(A_n^i)^\ell.
\]
Hence $\widehat A_n^i$ has image isomorphic to the image of
$(A_n^i)^\ell$, namely $T(n-2p^i)$. In particular,
$\widehat A_n^i$ is primitive.

The orthogonality relations involving distinct lower-summand blocks are
unchanged. More precisely, if a lower-summand idempotent $C$ is
orthogonal on both sides to $A_n^i$ and $(A_n^i)^\ell$, then it is also
orthogonal on both sides to
\[
  \widehat A_n^i
  =
  (A_n^i)^\ell-A_n^i(A_n^i)^\ell.
\]
Thus, together with the unchanged primitive idempotents $B_n^i$ for
$a_i>1$, the idempotents $T_n^i$ are mutually orthogonal.

Their sum is therefore an idempotent in $E_{n-1}\otimes\id_1$, and its
complement
\[
  E_n
  =
  E_{n-1}\otimes\id_1
  -
  \sum_{i=0}^{\tl(n)}T_n^i
\]
is again an idempotent.

The lower idempotents have images
\[
  \operatorname{im}(T_n^i)
  \cong
  T(n-2p^i)^{\oplus x_i},
\]
with the multiplicities $x_i$ stated above. By the Clebsch--Gordan rule,
the complementary summand is therefore isomorphic to $T(n)$. Hence
\[
  \operatorname{im}(E_n)\cong T(n),
\]
which gives the stated Tilting Jones--Wenzl recursion.
\end{proof}

\subsection{Derivation on foams} 
Fix a parameter $c \in \Bbbk$.
In \cite[Section 4.2]{QRSW2}, we introduced a family of $\mathfrak{sl}(2)$-actions on foams compatible with $\mathfrak{gl}(N)$-foam evaluation.  We review these formulas in the special case $N=2$. Define an operator $\dif_c$ on basic foams as follows. (In the setting of \cite[Section 4.2]{QRSW2}, parameters $t_1=c, t_2=1-c, t_3=\frac{1}{2}$, and the equivariant parameters are set to zero.)  For any $c \in \Bbbk$, set $\bar{c}=1-c$.  

\begin{subequations}
\begin{gather}
\label{eq:e-act-pol}  \dif_c\left(~\NB{\tikz[scale=1.5, font=\small]{\begin{scope}
  \draw (0,0) rectangle (1,1) coordinate [midway] (A); 
  \fill (A) circle (0.5mm) node[below] {};
\end{scope}}}~\right)
  = \ 
  0 \\
\label{eq:e-act-dig-cup}  \dif_c\left( ~\NB{\tikz[font=\tiny]{\begin{scope}
  \begin{scope}
    \coordinate (L) at (0,0);
    \coordinate (R) at (2,0);
    \coordinate (ML) at (0.5, 0);
    \coordinate (MR) at (1.5, 0);
    \draw[double] [->-] (L) -- (ML);
    \draw[double] [->-] (MR) -- (R); 
    \draw[->-] (ML).. controls + (0.4, 0.4) and +(-0.2, 0.4) .. (MR);
    \draw[->-] (ML).. controls + (0.2, -0.4) and +(-0.4, -0.4) .. (MR);
  \end{scope}  
 \begin{scope}[yshift = -1.3cm]
    \coordinate (LB) at (0,0);
    \coordinate (RB) at (2,0);
    \draw[double] [->-] (LB) -- (RB);
  \end{scope}  
  \draw[double]  (R) -- (RB);
  \draw[double]   (L) -- (LB);
  \draw[double] (ML) .. controls +(0, -1) and +(0, -1) .. (MR);
\end{scope}}}~\right)\ =\
 - c \cdot\ \NB{\tikz[font=\tiny]{\begin{scope}[font=\tiny]
  \begin{scope}
    \coordinate (L) at (0,0);
    \coordinate (R) at (2,0);
    \coordinate (ML) at (0.5, 0);
    \coordinate (MR) at (1.5, 0);
    \draw[double] [->-] (L) -- (ML);
    \draw[double] [->-] (MR) -- (R); 
    \draw[->-] (ML).. controls + (0.4, 0.4) and +(-0.2, 0.4) .. (MR)
    node[below, pos =0.3] {$\dotnewtoni[]$};
    \draw[->-] (ML).. controls +(0.2, -0.4) and +(-0.4, -0.4) .. (MR);
  \end{scope}  
 \begin{scope}[yshift = -1.3cm]
    \coordinate (LB) at (0,0);
    \coordinate (RB) at (2,0);
    \draw[double] [->-] (LB) -- (RB);
  \end{scope}  
  \draw[double]  (R) -- (RB);
  \draw[double]  (L) -- (LB);
  \draw[double] (ML) .. controls +(0, -1) and +(0, -1) .. (MR);
\end{scope}


}} 
  - \bar{c} \cdot \ \NB{\tikz[font=\tiny]{\begin{scope}[font=\tiny]
  \begin{scope}
    \coordinate (L) at (0,0);
    \coordinate (R) at (2,0);
    \coordinate (ML) at (0.5, 0);
    \coordinate (MR) at (1.5, 0);
    \draw[double] [->-] (L) -- (ML);
    \draw[->-][double]  (MR) -- (R); 
    \draw[->-] (ML).. controls + (0.4, 0.4) and +(-0.2, 0.4) .. (MR);
    \draw[->-] (ML).. controls +(0.2, -0.4) and +(-0.4, -0.4) .. (MR)
    node[below, pos =0.75] {$\dotnewtoni[]$};
  \end{scope}  
 \begin{scope}[yshift = -1.3cm]
    \coordinate (LB) at (0,0);
    \coordinate (RB) at (2,0);
    \draw[double] [->-] (LB) -- (RB);
  \end{scope}  
  \draw[double]  (R) -- (RB);
  \draw[double]  (L) -- (LB);
  \draw[double] (ML) .. controls +(0, -1) and +(0, -1) .. (MR);
\end{scope}

}}  \\
\label{eq:e-act-dig-cap}   \dif_c\left( ~\NB{\tikz[font=\tiny]{\begin{scope}
  \begin{scope}
    \coordinate (L) at (0,0);
    \coordinate (R) at (2,0);
    \coordinate (ML) at (0.5, 0);
    \coordinate (MR) at (1.5, 0);
    \draw[double] [->-] (L) -- (ML);
    \draw[double] [->-] (MR) -- (R); 
    \draw[->-] (ML).. controls + (0.4, 0.4) and +(-0.2, 0.4) .. (MR);
    \draw[->-] (ML).. controls + (0.2, -0.4) and +(-0.4, -0.4) .. (MR);
  \end{scope}  
 \begin{scope}[yshift = 1.3cm]
    \coordinate (LB) at (0,0);
    \coordinate (RB) at (2,0);
    \draw[double]  (LB) -- (RB);
  \end{scope}  
  \draw[double]  (R) -- (RB);
  \draw[double]  (L) -- (LB);
  \draw[double] (ML) .. controls +(0, 1) and +(0, 1) .. (MR);
\end{scope}}}~\right)\ =\
    -\bar{c}
  \cdot\ \NB{\tikz[font=\tiny]{\begin{scope}[font =\tiny]
  \begin{scope}
    \coordinate (L) at (0,0);
    \coordinate (R) at (2,0);
    \coordinate (ML) at (0.5, 0);
    \coordinate (MR) at (1.5, 0);
    \draw[double] [->-] (L) -- (ML);
    \draw[double] [->-] (MR) -- (R); 
    \draw[->-] (ML).. controls + (0.4, 0.4) and +(-0.2, 0.4) .. (MR)
    node[above, pos =0.3] {$\dotnewtoni[]$};
    \draw[->-] (ML).. controls + (0.2, -0.4) and +(-0.4, -0.4) .. (MR);
  \end{scope}  
 \begin{scope}[yshift = 1.3cm]
    \coordinate (LB) at (0,0);
    \coordinate (RB) at (2,0);
    \draw[double]  (LB) -- (RB);
  \end{scope}  
  \draw[double]  (R) -- (RB);
  \draw[double]  (L) -- (LB);
  \draw[double] (ML) .. controls +(0, 1) and +(0, 1) .. (MR);
\end{scope}

}}  
   - c
  \cdot \ \NB{\tikz[font=\tiny]{\begin{scope}[font =\tiny]
  \begin{scope}
    \coordinate (L) at (0,0);
    \coordinate (R) at (2,0);
    \coordinate (ML) at (0.5, 0);
    \coordinate (MR) at (1.5, 0);
    \draw[double] [->-] (L) -- (ML);
    \draw[double] [->-] (MR) -- (R); 
    \draw[->-] (ML).. controls + (0.4, 0.4) and +(-0.2, 0.4) .. (MR);
    \draw[->-] (ML).. controls + (0.2, -0.4) and +(-0.4, -0.4) .. (MR)
    node[above, pos =0.7, yshift = -0.5mm] {$\dotnewtoni[]$};
  \end{scope}  
 \begin{scope}[yshift = 1.3cm]
    \coordinate (LB) at (0,0);
    \coordinate (RB) at (2,0);
    \draw[double]  (LB) -- (RB);
  \end{scope}  
  \draw[double]  (R) -- (RB);
  \draw[double]  (L) -- (LB);
  \draw[double] (ML) .. controls +(0, 1) and +(0, 1) .. (MR);
\end{scope}

}}  \\
\label{eq:e-act-zip}   \dif_c\left(~ \NB{\tikz[font=\tiny]{\begin{scope}
  \begin{scope}
    \coordinate (L1) at (0.2,0.4);
    \coordinate (L2) at (0,0);
    \coordinate (R1) at (2.2,0.4);
    \coordinate (R2) at (2,0);
    \coordinate (ML) at (0.6, 0.2);
    \coordinate (MR) at (1.6, 0.2);
    \draw[double] [->-] (ML) -- (MR); 
    \draw (MR) .. controls +(0, 0) and +(-0.3,0) .. (R1) ;
    \draw (MR) .. controls +(0, 0) and +(-0.3,0) .. (R2);
    \draw (L1) .. controls +( 0.3, 0) and +(0,0) .. (ML);
    \draw (L2) .. controls +( 0.3, 0) and +(0,0) .. (ML);
  \end{scope}  
 \begin{scope}[yshift = -1cm]
    \coordinate (L1B) at (0.2,0.4);
    \coordinate (L2B) at (0,0);
    \coordinate (R1B) at (2.2,0.4);
    \coordinate (R2B) at (2,0);
    \draw[->-] (L1B) .. controls +( 0, 0) and +(0,0) .. (R1B); 
    \draw[->-] (L2B) .. controls +( 0, 0) and +(0,0) .. (R2B); 
 \end{scope}  
  \draw (R1) -- (R1B);
  \draw (R2) -- (R2B);
  \draw (L1) -- (L1B);
  \draw (L2) -- (L2B);
  \draw[double] (ML) .. controls +(0, -0.6) and +(0, -0.6) .. (MR);
\end{scope}

}}~\right)\ =\
   \bar{c} \cdot\ \NB{\tikz[font=\tiny]{\begin{scope}
  \begin{scope}
    \coordinate (L1) at (0.2,0.4);
    \coordinate (L2) at (0,0);
    \coordinate (R1) at (2.2,0.4);
    \coordinate (R2) at (2,0);
    \coordinate (ML) at (0.6, 0.2);
    \coordinate (MR) at (1.6, 0.2);
    \draw[double] [->-] (ML) -- (MR); 
    \draw (MR) .. controls +(0, 0) and +(-0.3,0) .. (R1) ;
    \draw (MR) .. controls +(0, 0) and +(-0.3,0) .. (R2);
    \draw (L1) .. controls +( 0.3, 0) and +(0,0) .. (ML);
    \draw (L2) .. controls +( 0.3, 0) and +(0,0) .. (ML);
  \end{scope}  
 \begin{scope}[yshift = -1cm]
    \coordinate (L1B) at (0.2,0.4);
    \coordinate (L2B) at (0,0);
    \coordinate (R1B) at (2.2,0.4);
    \coordinate (R2B) at (2,0);
    \draw[->-] (L1B) .. controls +( 0, 0) and +(0,0) .. (R1B); 
    \draw[->-] (L2B) .. controls +( 0, 0) and +(0,0) .. (R2B); 
 \end{scope}  
  \draw (R1) -- (R1B) node [pos = 0.2, left] {$\dotnewtoni[]$};
  \draw (R2) -- (R2B);
  \draw (L1) -- (L1B);
  \draw (L2) -- (L2B);
  \draw[double] (ML) .. controls +(0, -0.6) and +(0, -0.6) .. (MR);
\end{scope}

}} 
  +  c \cdot \ \NB{\tikz[font=\tiny]{\begin{scope}
  \begin{scope}
    \coordinate (L1) at (0.2,0.4);
    \coordinate (L2) at (0,0);
    \coordinate (R1) at (2.2,0.4);
    \coordinate (R2) at (2,0);
    \coordinate (ML) at (0.6, 0.2);
    \coordinate (MR) at (1.6, 0.2);
    \draw[double] [->-] (ML) -- (MR); 
    \draw (MR) .. controls +(0, 0) and +(-0.3,0) .. (R1) ;
    \draw (MR) .. controls +(0, 0) and +(-0.3,0) .. (R2);
    \draw (L1) .. controls +( 0.3, 0) and +(0,0) .. (ML);
    \draw (L2) .. controls +( 0.3, 0) and +(0,0) .. (ML);
  \end{scope}  
 \begin{scope}[yshift = -1cm]
    \coordinate (L1B) at (0.2,0.4);
    \coordinate (L2B) at (0,0);
    \coordinate (R1B) at (2.2,0.4);
    \coordinate (R2B) at (2,0);
    \draw[->-] (L1B) .. controls +( 0, 0) and +(0,0) .. (R1B); 
    \draw[->-] (L2B) .. controls +( 0, 0) and +(0,0) .. (R2B) 
    node [pos = 0.2, above] {$\dotnewtoni[]$};
 \end{scope}  
  \draw (R1) -- (R1B); 
  \draw (R2) -- (R2B);
  \draw (L1) -- (L1B);
  \draw (L2) -- (L2B);
  \draw[double] (ML) .. controls +(0, -0.6) and +(0, -0.6) .. (MR);
\end{scope}

}}  \\
\label{eq:e-act-unzip}   \dif_c\left(~ \NB{\tikz[font=\tiny]{\begin{scope}
  \begin{scope}
    \coordinate (L1) at (0.2,0.4);
    \coordinate (L2) at (0,0);
    \coordinate (R1) at (2.2,0.4);
    \coordinate (R2) at (2,0);
    \coordinate (ML) at (0.6, 0.2);
    \coordinate (MR) at (1.6, 0.2);
    \draw[double]  [->-] (ML) -- (MR); 
    \draw (MR) .. controls +(0, 0) and +(-0.3,0) .. (R1) ;
    \draw (MR) .. controls +(0, 0) and +(-0.3,0) .. (R2);
    \draw (L1) .. controls +( 0.3, 0) and +(0,0) .. (ML);
    \draw (L2) .. controls +( 0.3, 0) and +(0,0) .. (ML);
  \end{scope}  
 \begin{scope}[yshift = 1cm]
    \coordinate (L1B) at (0.2,0.4);
    \coordinate (L2B) at (0,0);
    \coordinate (R1B) at (2.2,0.4);
    \coordinate (R2B) at (2,0);
    \draw[->-] (L1B) .. controls +( 0, 0) and +(0,0) .. (R1B); 
    \draw[->-] (L2B) .. controls +( 0, 0) and +(0,0) .. (R2B); 
 \end{scope}  
  \draw (R1) -- (R1B);
  \draw (R2) -- (R2B);
  \draw (L1) -- (L1B);
  \draw (L2) -- (L2B);
  \draw[double] (ML) .. controls +(0, 0.6) and +(0, 0.6) .. (MR);
\end{scope}
}}~\right)\ =\
   c \cdot\ \NB{\tikz[font=\tiny]{\begin{scope}
  \begin{scope}
    \coordinate (L1) at (0.2,0.4);
    \coordinate (L2) at (0,0);
    \coordinate (R1) at (2.2,0.4);
    \coordinate (R2) at (2,0);
    \coordinate (ML) at (0.6, 0.2);
    \coordinate (MR) at (1.6, 0.2);
    \draw[double] [->-] (ML) -- (MR); 
    \draw (MR) .. controls +(0, 0) and +(-0.3,0) .. (R1) ;
    \draw (MR) .. controls +(0, 0) and +(-0.3,0) .. (R2);
    \draw (L1) .. controls +( 0.3, 0) and +(0,0) .. (ML);
    \draw (L2) .. controls +( 0.3, 0) and +(0,0) .. (ML); 
  \end{scope}  
 \begin{scope}[yshift = 1cm]
    \coordinate (L1B) at (0.2,0.4);
    \coordinate (L2B) at (0,0);
    \coordinate (R1B) at (2.2,0.4);
    \coordinate (R2B) at (2,0);
    \draw[->-] (L1B) .. controls +( 0, 0) and +(0,0) .. (R1B) node [left, pos
    = 0] 
  [pos=0.8, below] {$\dotnewtoni[]$};
    \draw[->-] (L2B) .. controls +( 0, 0) and +(0,0) .. (R2B); 
 \end{scope}  
  \draw (R1) -- (R1B);
  \draw (R2) -- (R2B);
  \draw (L1) -- (L1B);
  \draw (L2) -- (L2B);
  \draw[double] (ML) .. controls +(0, 0.6) and +(0, 0.6) .. (MR);
\end{scope}

}} 
   +  \bar{c} \cdot \ \NB{\tikz[font=\tiny]{\begin{scope}
  \begin{scope}
    \coordinate (L1) at (0.2,0.4);
    \coordinate (L2) at (0,0);
    \coordinate (R1) at (2.2,0.4);
    \coordinate (R2) at (2,0);
    \coordinate (ML) at (0.6, 0.2);
    \coordinate (MR) at (1.6, 0.2);
    \draw[double] [->-] (ML) -- (MR); 
    \draw (MR) .. controls +(0, 0) and +(-0.3,0) .. (R1) ;
    \draw (MR) .. controls +(0, 0) and +(-0.3,0) .. (R2);
    \draw (L1) .. controls +( 0.3, 0) and +(0,0) .. (ML);
    \draw (L2) .. controls +( 0.3, 0) and +(0,0) .. (ML) node [above,
    pos=0.2] {$\dotnewtoni[]$};
  \end{scope}  
 \begin{scope}[yshift = 1cm]
    \coordinate (L1B) at (0.2,0.4);
    \coordinate (L2B) at (0,0);
    \coordinate (R1B) at (2.2,0.4);
    \coordinate (R2B) at (2,0);
    \draw[->-] (L1B) .. controls +( 0, 0) and +(0,0) .. (R1B); 
    \draw[->-] (L2B) .. controls +( 0, 0) and +(0,0) .. (R2B); 
 \end{scope}  
  \draw (R1) -- (R1B);
  \draw (R2) -- (R2B);
  \draw (L1) -- (L1B);
  \draw (L2) -- (L2B);
  \draw[double] (ML) .. controls +(0, 0.6) and +(0, 0.6) .. (MR);
\end{scope}

}}    \\
\label{eq:e-act-cup}   \dif_c\left( \NB{\tikz[font=\tiny, scale=1.2]{\begin{scope}
  \draw (0,0) arc (180 :0: 0.5cm and 0.2cm); 
  \draw[very thin] (0,0) arc (180 :0: 0.5cm and -0.6cm) node[pos=0.5,
  above] {};
  \draw (0,0) arc (180 :0: 0.5cm and -0.2cm);
\end{scope}

}}\right)\ =\
 \dif_c\left( \NB{\tikz[font=\tiny, scale=1.2]{\begin{scope}[-]
  \draw (0,0) arc (180 :0: 0.5cm and 0.2cm);
  \draw[very thin] (0,0) arc (180 :0: 0.5cm and 0.6cm) node[pos=0.5,
  below] {};
  \draw (0,0) arc (180 :0: 0.5cm and -0.2cm);
\end{scope}

}}\right)\ =\
 \dif_c\left( \NB{\tikz[font=\tiny, scale=1.2]{\begin{scope}[scale=0.6]
\tdplotsetmaincoords{70}{25}
\begin{scope}[scale = 1.5, tdplot_main_coords]
  \tikzset{yxplane/.style={canvas is xy plane at z=#1}}
  \begin{scope}[yxplane=1]
    \coordinate (AT) at ({cos(  45)}, {sin(  45)});
    \coordinate (BT) at ({cos(135)}, {sin(135)});
    \coordinate (CT) at ({cos(225)}, {sin(225)});
    \coordinate (DT) at ({cos(315)}, {sin(315)});
    \coordinate (aT) at (  0.3, 0);
    \coordinate (bT) at ( -0.3, 0);
    \draw (AT) .. controls (aT) and (bT) .. (BT)
    coordinate[pos=0.5] (eT); 
    \draw (DT) .. controls (aT) and (bT) .. (CT) coordinate[pos=0.5] (fT);
  \end{scope}

  \begin{scope}[yxplane=0]
    \coordinate (AM) at ({cos(  45)}, {sin(  45)});
    \coordinate (BM) at ({cos(135)}, {sin(135)});
    \coordinate (CM) at ({cos(225)}, {sin(225)});
    \coordinate (DM) at ({cos(315)}, {sin(315)});
    \coordinate (aM) at (0, 0.3);
    \coordinate (bM) at (0, -0.3);
    \draw (AM) .. controls (aM) and (bM) .. (DM) coordinate[pos=0.5] (eM);
    \draw (BM) .. controls (aM) and (bM) .. (CM) coordinate[pos=0.5] (fM);
\end{scope}
  \coordinate (OT) at (0,0, 0.5);
\draw (AM) -- (AT);
  \draw (BM) -- (BT);
  \draw (CM) -- (CT);
  \draw (DM) -- (DT);
  \draw[densely dotted] (eM) ..controls +(0,0,0.2) and + (0.1,0,0).. (OT);
  \draw[densely dotted] (fM) ..controls +(0,0,0.2) and + (-0.1,0,0).. (OT);
  \draw[densely dotted] (eT) ..controls +(0,0,-0.2) and + (0,0.1,0).. (OT);
  \draw[densely dotted] (fT) ..controls +(0,0,-0.2) and + (0,-0.1,0).. (OT);
\end{scope}  
\end{scope}

}}\right)\ =\
0
\end{gather}
\end{subequations}
Extend $\dif_c$ to a derivation on all $\mathfrak{gl}(2)$-foams by the Leibniz rule.

\subsection{Twists}
 It is possible to twist the $\dif_c$-action on the state space of a web.
This idea goes back to \cite{KRWitt} where
  the Witt-type action was twisted in order to make differentials
  occurring in link homology equivariant with respect to the Witt
  action.  Here we review this twisting
  in the simplified setting of an $H$-action.  
  For more details, see \cite[Section 3.3]{QRSW3}.

    A \emph{green-dotted} web is a web endowed with
    a finite collection $D$ of \emph{green dots}, i.e., marked
    points with multiplicities in $\Bbbk$ located in the interior
    of edges of the web.
    The action of $\dif_c$ on a facet adjacent to an edge carrying a green dot is twisted as follows:
\[
\label{eq:f-act-pol-twisthollow} \dif_c\left(\NB{\tikz[scale=1.5, font=\small]{\begin{scope}
  \draw[thin] (0,0.5) rectangle (1,1);
  \draw[thick] (0,1) -- (1,1)  coordinate [pos=0.4] (A);
  \filldraw[draw= green!50!black, fill = white] (A) circle (.5mm)
  node[below,  green!50!black] {$\lambda$};
\end{scope}}}\right) 
  =\lambda\cdot\  \NB{\tikz[scale=1.5, font=\small]{\begin{scope}
  \draw[thin] (0,0.5) rectangle (1,1);
  \draw[thick] (0,1) -- (1,1)  coordinate [pos=0.4] (A);
  \filldraw[draw= green!50!black, fill = white] (A) circle (.5mm)
  node[below,  green!50!black] {$\lambda$};
  \node (B) at (.75,.75) {$\dotnewtoni[]$};
\end{scope}

}} 
\]

 \section{Uncolored link homology at a root of unity}
In this section we construct a functorial categorification of the Jones polynomial at a $2p$th root of unity by equipping Khovanov homology with a $p$-differential.  
The results of this section are direct consequences of \cite{QRSW2, QRSW3}.

\subsection{Derivations on Khovanov homology}
An $\mathfrak{sl}(2)$-action on equivariant Khovanov--Rozansky homology was constructed in \cite{QRSW3}. In this preprint we use part of this action (we set $\dif_c$ to be the $f$-operator in \cite{QRSW3}) in the specialized case where $N=2$, the base ring is $\Bbbk$, a field of characteristic $p$ and the equivariant parameters are zero.

We refer the reader to \cite[Theorem 4.3]{QRSW3} to see how the action of $\mathfrak{sl}(2)$ is defined in the equivariant setting, and \cite[Section 6.1]{QRSW3} when the equivariant parameters are set to zero in characteristic $p$.  The green dots below indicate how the $\dif_c$-action is twisted on the state spaces associated to webs.

Khovanov homology is built out of resolving the two crossing diagrams:
\begin{subequations}
\begin{equation}
  T= \NB{\tikz[xscale = 0.6]{\begin{scope}[font=\tiny]
  \draw[->] (0.5, -0.5) ..controls +(0,0.3) and +(0,-0.3) .. (-0.5,
  0.5) node[pos=1, above] {} coordinate[pos =0.2] (t1);
  \fill[white] (0,0) circle (2mm);
  \draw[->] (-0.5, -0.5) ..controls +(0,0.3) and +(0,-0.3) .. (0.5,
  0.5) node[pos=1, above] {} coordinate[pos =0.2] (t2);
\end{scope}}} :=
\left(
    \NB{  \tikz[xscale = 3.5, yscale = 3]{
    \node (i0) at (0, 0) { $q^{-1}$ \NB{\tikz[font= \tiny,
  scale=0.6]{\begin{scope}
  \coordinate (bl) at (-0.5, -1);
  \coordinate (br) at ( 0.5, -1);
  \coordinate (tl) at (-0.5,  1);
  \coordinate (tr) at ( 0.5,  1);
    \coordinate (ml) at (-0.5,  -.8);
        \coordinate (Ml) at (-0.5,  .8);
 \coordinate (mr) at (0.5,  -.6);
\coordinate (Mr) at (0.5,  .6);
 
   \draw[>->] (bl) -- (tl) 
   node[below, pos
  =0] {} coordinate[pos = 0.5] (g1);
  \filldraw[draw= green!50!black, fill = white] (g1) circle (1mm)
  node[left, green!50!black] {$-c$};
  
  
    \draw[>->] (br) -- (tr) 
    node[below, pos
  =0] {} coordinate[pos = 0.5] (g2);
  \filldraw[draw= green!50!black, fill = white] (g2) circle (1mm)
  node[right, green!50!black] {$-\bar{c} $};
  

\end{scope}}} };
     \node (i1) at (-1, 0) {\NB{\tikz[font= \tiny,
  scale=0.6]{\begin{scope}
  \coordinate (bl) at (-0.5, -1);
  \coordinate (br) at ( 0.5, -1);
  \coordinate (bm) at (  0,-0.3);
  \coordinate (tl) at (-0.5,  1);
  \coordinate (tr) at ( 0.5,  1);
  \coordinate (tm) at (  0, 0.3);
  \draw[>-]  (bl) .. controls +( 0, 0.5) and +(0,0) .. (bm);
  \draw[>-]  (br) .. controls +( 0, 0.5) and +(0,0) .. (bm);
  \draw[<-]  (tl) .. controls +( 0, -0.5) and +(0,0) .. (tm);

  \draw[<-]  (tr) .. controls +( 0, -0.5) and +(0,0) .. (tm);
  \draw[double] (bm) -- (tm); 
 
\end{scope}}}[1] };
\draw[->] (i1) -- (i0) coordinate[pos=0.5] (a);
\node[above] at (a) {\NB{\tikz[font=\tiny, scale=.5]{\begin{scope}
  \begin{scope}
    \coordinate (L1) at (0.2,0.4);
    \coordinate (L2) at (0,0);
    \coordinate (R1) at (2.2,0.4);
    \coordinate (R2) at (2,0);
    \coordinate (ML) at (0.6, 0.2);
    \coordinate (MR) at (1.6, 0.2);
    \draw[double] (ML) -- (MR);
    \draw (MR) .. controls +(0, 0) and +(-0.3,0) .. (R1) ;
    \draw (MR) .. controls +(0, 0) and +(-0.3,0) .. (R2);
    \draw (L1) .. controls +( 0.3, 0) and +(0,0) .. (ML);
    \draw (L2) .. controls +( 0.3, 0) and +(0,0) .. (ML);
  \end{scope}  
 \begin{scope}[yshift = 1cm]
    \coordinate (L1B) at (0.2,0.4);
    \coordinate (L2B) at (0,0);
    \coordinate (R1B) at (2.2,0.4);
    \coordinate (R2B) at (2,0);
    \draw (L1B) .. controls +( 0, 0) and +(0,0) .. (R1B); 
    \draw (L2B) .. controls +( 0, 0) and +(0,0) .. (R2B);
 \end{scope}  
  \draw (R1) -- (R1B);
  \draw (R2) -- (R2B);
  \draw (L1) -- (L1B);
  \draw (L2) -- (L2B);
  \draw[double] (ML) .. controls +(0, 0.6) and +(0, 0.6) .. (MR);
\end{scope}}}};
  }}
\right) \ ,
    \end{equation}
    \begin{equation}
  T^{-1}= \NB{\tikz[xscale = 0.6]{\begin{scope}[font=\tiny]
  \draw[->] (-0.5, -0.5) ..controls +(0,0.3) and +(0,-0.3) .. (0.5,
  0.5);
  \fill[white] (0,0) circle (2mm);
  \draw[->] (0.5, -0.5) ..controls +(0,0.3) and +(0,-0.3) .. (-0.5,
  0.5);
\end{scope}}}:=
\left(  \NB{\tikz[xscale = 3.5, yscale = 3]{
    \node (i0) at (-1, 0) {  $q$\ \NB{\tikz[font= \tiny,
  scale=0.6]{\begin{scope}
  \coordinate (bl) at (-0.5, -1);
  \coordinate (br) at ( 0.5, -1);
  \coordinate (tl) at (-0.5,  1);
  \coordinate (tr) at ( 0.5,  1);
    \coordinate (ml) at (-0.5,  -.8);
        \coordinate (Ml) at (-0.5,  .8);
 \coordinate (mr) at (0.5,  -.6);
\coordinate (Mr) at (0.5,  .6);
 
   \draw[>->] (bl) -- (tl); 
  \filldraw[draw= green!50!black, fill = white] (g1) circle (1mm)
  node[left, green!50!black] {$\bar{c}$};
  
  
    \draw[>->] (br) -- (tr);
  \filldraw[draw= green!50!black, fill = white] (g2) circle (1mm)
  node[right, green!50!black] {$c$};
  

\end{scope}}} };
     \node (i1) at (0, 0) {  \NB{\tikz[font= \tiny,
  scale=0.6]{}}[-1] };
\draw[->] (i0) -- (i1) coordinate[pos=0.5] (b);
\node[above] at (b) {\NB{\tikz[font=\tiny, scale=.5]{\begin{scope}
  \begin{scope}
    \coordinate (L1) at (0.2,0.4);
    \coordinate (L2) at (0,0);
    \coordinate (R1) at (2.2,0.4);
    \coordinate (R2) at (2,0);
    \coordinate (ML) at (0.6, 0.2);
    \coordinate (MR) at (1.6, 0.2);
    \draw[double] (ML) -- (MR);
    \draw (MR) .. controls +(0, 0) and +(-0.3,0) .. (R1) ;
    \draw (MR) .. controls +(0, 0) and +(-0.3,0) .. (R2);
    \draw (L1) .. controls +( 0.3, 0) and +(0,0) .. (ML);
    \draw (L2) .. controls +( 0.3, 0) and +(0,0) .. (ML);
  \end{scope}  
 \begin{scope}[yshift = -1cm]
    \coordinate (L1B) at (0.2,0.4);
    \coordinate (L2B) at (0,0);
    \coordinate (R1B) at (2.2,0.4);
    \coordinate (R2B) at (2,0);
    \draw (L1B) .. controls +( 0, 0) and +(0,0) .. (R1B);
    \draw (L2B) .. controls +( 0, 0) and +(0,0) .. (R2B); 
 \end{scope}  
  \draw (R1) -- (R1B);
  \draw (R2) -- (R2B);
  \draw (L1) -- (L1B);
  \draw (L2) -- (L2B);
  \draw[double] (ML) .. controls +(0, -0.6) and +(0, -0.6) .. (MR);
\end{scope}

}}}; 
      }} 
      \right)
  \ .
    \end{equation}
  \end{subequations}
Unlike the convention in \cite{QRSW2}, we place the terms $\NB{\tikz[font= \tiny,
  scale=0.4]{}}$  in cohomological degrees $t=- 1$ and $t=1$ respectively. 

\begin{defn}
Let $L$ be a link with a diagram $D$. 
\begin{enumerate}
    \item Set $\CKh^{\dif_c}(D) $ to be the Khovanov chain complex of $D$ with coefficients in a field $\Bbbk$ of characteristic $p$, equipped with the action of the graded Hopf algebra $H=\Bbbk[\dif_c]/(\dif_c^p)$. Here, as the relative homotopy category of $\Bbbk$-vector spaces is equivalent to the derived category of graded $H$-modules, 
$\CKh^{\dif_c}(D)$ is regarded as an object in $\mc{D}(H,d)$.
\item  Via Theorem \ref{thm-Rickard}, the image of $\CKh^{\dif_c}(D)$ in $\mc{D}(H,d)/\mc{P}_H=H\udmod$ will be referred to as the \emph{slash Khovanov homology of $D$}, and will be denoted by $\Kh_/^{\dif_c}(D)$.
\end{enumerate}
\end{defn}

To simplify notation, we will draw (pieces of) link diagrams (with green dots) 
to represent the $\dif_c$-equivariant Khovanov chain complex of these diagrams. In other words, we will mostly drop
$\CKh^{\dif_c}(\cdot)$ and $\Kh_/^{\dif_c}(\cdot)$ around diagrams to prevent
overloaded figures.

\begin{thm} \cite[Theorem 6.8]{QRSW3} \label{thm:uncolorKh} \begin{enumerate}
    \item[(i)] If $D_1$ and $D_2$ are two planar projections of a framed link $L$, then their associated chain complexes of graded $H$-modules are isomorphic in the derived category $\mc{D}(H,d)$. Denote the object by $\CKh^{\dif_c}(L)$.
    \item[(ii)] The image of $\CKh^{\dif_c}(L)$ in the stable category of $H$, which will be referred to as the slash Khovanov homology of $L$ and denoted $\Kh_/^{\dif_c}(L)$,  is a framed link invariant valued in $H\udmod$.  
    \item[(iii)] For a link $L$, $\chi_{\mathbb{O}}(\CKh^{\dif_c}(L))=\chi_{\mathbb{O}}(\Kh^{\dif_c}_/(L))$ is equal to the Jones polynomial of $L$ evaluated at a $2p$th root of unity.
    \end{enumerate}
\end{thm}

\subsection{Examples}

\begin{example}\label{eg-unknot}
As an $H$-module, the chain complex of the unknot $\ovoid$ is equal to:
\begin{equation}
\CKh^{\dif_c}(\ovoid)  \cong 
q^{-1} U_0 \oplus q U_0.
\end{equation}
This is computed as in \cite[Section 7.1]{QRSW3} where one makes the necessary modifications in characteristic $p$, while also killing the equivariant parameters. More explicitly, in the language of foams (as in \cite{QRSW3}), the invariant is linearly spanned by a cup with $i$ dots for $0 \leq i \leq 1$:
\[
\CKh^{\dif_c}(\ovoid)=\Bbbk\left\langle v_i=\NB{\tikz[]{\begin{scope}
  \draw (0,0) arc (180 :0: 0.5cm and 0.2cm);
  \draw[very thin] (0,0) arc (180 :0: 0.5cm and -0.6cm) node [pos=0.5,
  above] {$\bullet~i$};
  \draw (0,0) arc (180 :0: 0.5cm and -0.2cm);
\end{scope}

}}~\Big|~ i=0,1 \right\rangle.
\]
\end{example}

\begin{example} \label{eg-hopf-link}
As in \cite[Equation 7.5]{QRSW1} or \cite[Proposition 5.1]{Roz}, for the Hopf link $T_{2,2}$, the complex in the relative homotopy category is given by 
\begin{equation} \label{eq:hopfcomplex}
  \NB{\tikz[xscale = 3, yscale = 3]{
       \node (i0) at (0, 0) { \NB{\tikz[font= \tiny,
  scale=0.6]{\begin{scope}
  \coordinate (bl) at (-0.5, -1);
  \coordinate (br) at ( 0.5, -1);
  \coordinate (tl) at (-0.5,  1);
  \coordinate (tr) at ( 0.5,  1);
    \coordinate (ml) at (-0.5,  -.8);
        \coordinate (Ml) at (-0.5,  .8);
 \coordinate (mr) at (0.5,  -.6);
\coordinate (Mr) at (0.5,  .6);

 
   \draw[->] (bl) -- (tl) node[pos = 0, below] {} node[pos = 1,
  above] {} coordinate[pos = 0.4] (ga);
  
  
    \draw[->] (br) -- (tr) node[pos = 0, below] {} node[pos = 1, above] {} coordinate[pos = 0.4] (gb);
  
  
    \draw (.5, -1) arc (-180:0:0.5) -- (1.5,1) arc (-180:0:-0.5) coordinate[pos = 0.25] (X) coordinate[pos = 0.75] (Y) coordinate[pos = 0.5] (gb);
    
       \draw (-.5, -1) arc (180:0:-0.5) -- (-1.5,1) arc (180:0:0.5) coordinate[pos = 0.25] (XX) coordinate[pos = 0.75] (YY) coordinate[pos = 0.5] (ga);

          \filldraw[draw= green!50!black, fill = white] (ga) circle (1mm)
  node[above, green!50!black] {$-1$};

     \filldraw[draw= green!50!black, fill = white] (gb) circle (1mm)
  node[above, green!50!black] {$-1$};

\end{scope}}} };
\node (i2) at (-1, 0) { \NB{\tikz[font= \tiny,
  scale=0.6]{\begin{scope}
  \coordinate (bl) at (-0.5, -1);
  \coordinate (br) at ( 0.5, -1);
  \coordinate (bm) at (  0,-0.3);
  \coordinate (tl) at (-0.5,  1);
  \coordinate (tr) at ( 0.5,  1);
  \coordinate (tm) at (  0, 0.3);
  \draw[>-]  (bl) .. controls +( 0, 0.5) and +(0,0) .. (bm)
  node[right, pos = 0] {};
  \draw[>-]  (br) .. controls +( 0, 0.5) and +(0,0) .. (bm)
  node[left, pos = 0] {};
  \draw[<-]  (tl) .. controls +( 0, -0.5) and +(0,0) .. (tm)
  node[right, pos = 0] {} coordinate[pos = 0.25] (ga) ;
    \filldraw[draw= green!50!black, fill = white] (ga) circle (1mm)
  node[left, green!50!black] {$c-1$};

  \draw[<-]  (tr) .. controls +( 0, -0.5) and +(0,0) .. (tm)
  node[left, pos = 0] {} coordinate[pos = 0.25] (gb) ;
    \filldraw[draw= green!50!black, fill = white] (gb) circle (1mm)
  node[right, green!50!black] {$-c$};
  \draw [double] (bm) -- (tm) node[left, pos = 0.5] {};
  \draw (.5, 1) arc (180:0:0.5) -- (1.5,-1) arc (180:0:-0.5);
  \draw (-.5, 1) arc (-180:0:-0.5) -- (-1.5,-1) arc (-180:0:0.5);
\end{scope}}} };
      \node (i3) at (-2, 0) { \NB{\tikz[font= \tiny,
  scale=0.6]{\begin{scope}
  \coordinate (bl) at (-0.5, -1);
  \coordinate (br) at ( 0.5, -1);
  \coordinate (bm) at (  0,-0.3);
  \coordinate (tl) at (-0.5,  1);
  \coordinate (tr) at ( 0.5,  1);
  \coordinate (tm) at (  0, 0.3);
  \draw[>-]  (bl) .. controls +( 0, 0.5) and +(0,0) .. (bm)
  node[right, pos = 0] {};
  \draw[>-]  (br) .. controls +( 0, 0.5) and +(0,0) .. (bm)
  node[left, pos = 0] {};
  \draw[<-]  (tl) .. controls +( 0, -0.5) and +(0,0) .. (tm)
  node[right, pos = 0] {} coordinate[pos = 0.25] (ga) ;
    \filldraw[draw= green!50!black, fill = white] (ga) circle (1mm)
  node[left, green!50!black] {$c$};

  \draw[<-]  (tr) .. controls +( 0, -0.5) and +(0,0) .. (tm)
  node[left, pos = 0] {} coordinate[pos = 0.25] (gb) ;
    \filldraw[draw= green!50!black, fill = white] (gb) circle (1mm)
  node[right, green!50!black] {$1-c$};
  \draw [double] (bm) -- (tm) node[left, pos = 0.5] {};
  \draw (.5, 1) arc (180:0:0.5) -- (1.5,-1) arc (180:0:-0.5);
  \draw (-.5, 1) arc (-180:0:-0.5) -- (-1.5,-1) arc (-180:0:0.5);
\end{scope}}} };
\draw[->] (i3) -- (i2) node[pos=0.5, above] {$0$};
\draw[->] (i2) -- (i0) node[pos=0.5, above] {$\mapH$};
  }} 
\end{equation}
where $\eta$ is given by the closure of unzip with two dots on each facet with different signs:
\[
\label{eq:etafef} \eta\left(
\NB{\tikz[font= \tiny, scale=0.6]{}}\right) 
  = 
\NB{\tikz[font= \tiny, scale=0.6]{\begin{scope}
  \coordinate (bl) at (-0.5, -1);
  \coordinate (br) at ( 0.5, -1);
        \coordinate (ml) at (-1.5, 0);
  \coordinate (tl) at (-0.5,  1);
  \coordinate (tr) at ( 0.5,  1);
    \coordinate (ml) at (-1.5,  0);
        \coordinate (Ml) at (-0.5,  .8);
 \coordinate (mr) at (0.5,  -.6);
\coordinate (Mr) at (0.5,  .6);

 
   \draw[->] (bl) -- (tl) node[pos = 0, below] {} node[pos = 1,
  above] {} coordinate[pos = 0.4] (ga);
  
  
    \draw[->] (br) -- (tr) node[pos = 0, below] {} node[pos = 1, above] {} coordinate[pos = 0.4] (gb);
  
  
    \draw (.5, -1) arc (-180:0:0.5) -- (1.5,1) arc (-180:0:-0.5) coordinate[pos = 0.25] (X) coordinate[pos = 0.75] (Y) coordinate[pos = 0.5] (gb);
    
       \draw (-.5, -1) arc (180:0:-0.5) -- (-1.5,1) arc (180:0:0.5) coordinate[pos = 0.25] (XX) coordinate[pos = 0.75] (YY) coordinate[pos = 0.5] (ga);

          \filldraw[draw= green!50!black, fill = white] (ga) circle (1mm)
  node[above, green!50!black] {$-1$};

     \filldraw[draw= green!50!black, fill = white] (gb) circle (1mm)
  node[above, green!50!black] {$-1$};
     \filldraw[draw= black!50!black, fill = black] (ml) circle (1mm)
  node[left, black!50!black] {};

\end{scope}}} 
  -
\NB{\tikz[font= \tiny, scale=0.6]{\begin{scope}
  \coordinate (bl) at (-0.5, -1);
  \coordinate (br) at ( 0.5, -1);
        \coordinate (ml) at (1.5, 0);
  \coordinate (tl) at (-0.5,  1);
  \coordinate (tr) at ( 0.5,  1);
        \coordinate (Ml) at (-0.5,  .8);
 \coordinate (mr) at (.5,  -.6);
\coordinate (Mr) at (0.5,  .6);

 
   \draw[->] (bl) -- (tl) node[pos = 0, below] {} node[pos = 1,
  above] {} coordinate[pos = 0.4] (ga);
  
  
    \draw[->] (br) -- (tr) node[pos = 0, below] {} node[pos = 1, above] {} coordinate[pos = 0.4] (gb);
  
  
    \draw (.5, -1) arc (-180:0:0.5) -- (1.5,1) arc (-180:0:-0.5) coordinate[pos = 0.25] (X) coordinate[pos = 0.75] (Y) coordinate[pos = 0.5] (gb);
    
       \draw (-.5, -1) arc (180:0:-0.5) -- (-1.5,1) arc (180:0:0.5) coordinate[pos = 0.25] (XX) coordinate[pos = 0.75] (YY) coordinate[pos = 0.5] (ga);

          \filldraw[draw= green!50!black, fill = white] (ga) circle (1mm)
  node[above, green!50!black] {$-1$};

     \filldraw[draw= green!50!black, fill = white] (gb) circle (1mm)
  node[above, green!50!black] {$-1$};
     \filldraw[draw= black!50!black, fill = black] (ml) circle (1mm)
  node[left, black!50!black] {};

\end{scope}}} .
\]

The leftmost and the middle terms are both isotopic to the $\Theta$-web, while the rightmost term is a disjoint union of two circles. The middle and rightmost terms in \eqref{eq:hopfcomplex} are spanned by:
\begin{equation}
v_1= \NB{\tikz[font= \tiny, scale=0.6]{}}, \quad
v_2=  \NB{\tikz[font= \tiny,scale=0.6]{\begin{scope}
  \coordinate (bl) at (-0.5, -1);
      \coordinate (ml) at (-1.5, 0);
  \coordinate (br) at ( 0.5, -1);
  \coordinate (bm) at (  0,-0.3);
  \coordinate (tl) at (-0.5,  1);
  \coordinate (tr) at ( 0.5,  1);
  \coordinate (tm) at (  0, 0.3);
  \draw[>-]  (bl) .. controls +( 0, 0.5) and +(0,0) .. (bm)
  node[right, pos = 0] {};
  \draw[>-]  (br) .. controls +( 0, 0.5) and +(0,0) .. (bm)
  node[left, pos = 0] {};
  \draw[<-]  (tl) .. controls +( 0, -0.5) and +(0,0) .. (tm)
  node[right, pos = 0] {} coordinate[pos = 0.25] (ga) ;
    \filldraw[draw= green!50!black, fill = white] (ga) circle (1mm)
  node[left, green!50!black] {$c-1$};
     \filldraw[draw= black!50!black, fill = black] (ml) circle (1mm)
  node[left, black!50!black] {};
  
  \draw[<-]  (tr) .. controls +( 0, -0.5) and +(0,0) .. (tm)
  node[left, pos = 0] {} coordinate[pos = 0.25] (gb) ;
    \filldraw[draw= green!50!black, fill = white] (gb) circle (1mm)
  node[right, green!50!black] {$-c$};
  \draw [double] (bm) -- (tm) node[left, pos = 0.5] {};
  \draw (.5, 1) arc (180:0:0.5) -- (1.5,-1) arc (180:0:-0.5);
  \draw (-.5, 1) arc (-180:0:-0.5) -- (-1.5,-1) arc (-180:0:0.5);
\end{scope}}}, 
\end{equation}
\begin{equation}
v_3=  \NB{\tikz[font= \tiny, scale=0.6]{}}, \quad
v_4= \NB{\tikz[font= \tiny, scale=0.6]{}}, \quad
v_5=  \NB{\tikz[font= \tiny,scale=0.6]{}}, \quad
v_6=  \NB{\tikz[font= \tiny, scale=0.6]{\begin{scope}
  \coordinate (bl) at (-0.5, -1);
  \coordinate (br) at ( 0.5, -1);
        \coordinate (ml) at (-1.5, 0);
  \coordinate (tl) at (-0.5,  1);
  \coordinate (tr) at ( 0.5,  1);
    \coordinate (ml) at (-1.5,  0);
        \coordinate (Ml) at (-0.5,  .8);
 \coordinate (mr) at (1.5,  0);
\coordinate (Mr) at (0.5,  .6);

 
   \draw[->] (bl) -- (tl) node[pos = 0, below] {} node[pos = 1,
  above] {} coordinate[pos = 0.4] (ga);
  
  
    \draw[->] (br) -- (tr) node[pos = 0, below] {} node[pos = 1, above] {} coordinate[pos = 0.4] (gb);
  
  
    \draw (.5, -1) arc (-180:0:0.5) -- (1.5,1) arc (-180:0:-0.5) coordinate[pos = 0.25] (X) coordinate[pos = 0.75] (Y) coordinate[pos = 0.5] (gb);
    
       \draw (-.5, -1) arc (180:0:-0.5) -- (-1.5,1) arc (180:0:0.5) coordinate[pos = 0.25] (XX) coordinate[pos = 0.75] (YY) coordinate[pos = 0.5] (ga);

          \filldraw[draw= green!50!black, fill = white] (ga) circle (1mm)
  node[above, green!50!black] {$-1$};

     \filldraw[draw= green!50!black, fill = white] (gb) circle (1mm)
  node[above, green!50!black] {$-1$};
     \filldraw[draw= black!50!black, fill = black] (ml) circle (1mm)
  node[left, black!50!black] {};
       \filldraw[draw= black!50!black, fill = black] (mr) circle (1mm)
  node[left, black!50!black] {};

\end{scope}}} .
\end{equation}
It is straightforward to calculate:
\[
\eta(v_1)=v_4-v_5, \quad \eta(v_2)=-v_6, \]
\[
\dif_c(v_3)=- v_4-v_5, \quad
\dif_c(v_4)=-v_6, \quad
\dif_c(v_5)=- v_6
\]
and all other mappings on these basis vectors are zero. Set
\[
w_3:=v_3, \quad w_4:=-v_4-v_5, \quad w_5:=v_4-v_5, \quad w_6:=-v_6.
\]
Then we have
\[
\eta(v_1)=w_5, \quad \eta(v_2)=w_6, \]
\[
\dif_c(w_3)=w_4, \quad
\dif_c(w_4)=-2w_6, 
\]
and all other mappings on these basis vectors are zero.
Thus, inside the summand
\begin{equation} \label{eq:hopfcomplex-2}
  \NB{\tikz[xscale = 3, yscale = 3]{
       \node (i0) at (0, 0) { \NB{\tikz[font= \tiny,
  scale=0.6]{}} };
\node (i2) at (-1, 0) { \NB{\tikz[font= \tiny,
  scale=0.6]{}} };
\draw[->] (i2) -- (i0) node[pos=0.5, above] {$\mapH$};
  }} 
\end{equation}
the $H$-module $\Bbbk\langle v_1,v_2,w_5,w_6\rangle$ constitutes a contractible subcomplex. The quotient of \eqref{eq:hopfcomplex-2} by this subcomplex is isomorphic to $U_1$ sitting in homological degree zero spanned by $w_3$ and $w_4$.

The leftmost $H$-summand in \eqref{eq:hopfcomplex} has $p$-DG structure
\begin{equation}
  \NB{\tikz[xscale = 3, yscale = 3]{
    \node (i2) at (-1, 0) {\NB{\tikz[font= \tiny,
  scale=0.6]{\begin{scope}
  \coordinate (bl) at (-0.5, -1);
      \coordinate (ml) at (-1.5, 0);
  \coordinate (br) at ( 0.5, -1);
  \coordinate (bm) at (  0,-0.3);
  \coordinate (tl) at (-0.5,  1);
  \coordinate (tr) at ( 0.5,  1);
  \coordinate (tm) at (  0, 0.3);
  \draw[>-]  (bl) .. controls +( 0, 0.5) and +(0,0) .. (bm)
  node[right, pos = 0] {};
  \draw[>-]  (br) .. controls +( 0, 0.5) and +(0,0) .. (bm)
  node[left, pos = 0] {};
  \draw[<-]  (tl) .. controls +( 0, -0.5) and +(0,0) .. (tm)
  node[right, pos = 0] {} coordinate[pos = 0.25] (ga) ;
    \filldraw[draw= green!50!black, fill = white] (ga) circle (1mm)
  node[left, green!50!black] {$c$};
     \filldraw[draw= black!50!black, fill = black] (ml) circle (1mm)
  node[left, black!50!black] {};
  
  \draw[<-]  (tr) .. controls +( 0, -0.5) and +(0,0) .. (tm)
  node[left, pos = 0] {} coordinate[pos = 0.25] (gb) ;
    \filldraw[draw= green!50!black, fill = white] (gb) circle (1mm)
  node[right, green!50!black] {$1-c$};
  \draw [double] (bm) -- (tm) node[left, pos = 0.5] {};
  \draw (.5, 1) arc (180:0:0.5) -- (1.5,-1) arc (180:0:-0.5);
  \draw (-.5, 1) arc (-180:0:-0.5) -- (-1.5,-1) arc (-180:0:0.5);
\end{scope}}} };
      \node (i3) at (-2, 0) { \NB{\tikz[font= \tiny,
  scale=0.6]{}} };
    \draw[->] (i3) -- (i2) node[pos=0.5, above] {$0$};
  }} 
\end{equation}
It follows that the complex $\CKh^{\dif_c}(T_{2,2})$ is isomorphic, in the relative homotopy category, to a direct sum: 
\[
\CKh^{\dif_c}(T_{2,2}) \cong
t^{-2}( U_0 \oplus q^2 U_0 ) \oplus q^{-4}  U_1  .
\]
This computation is done as in \cite{Roz}, setting the equivariant parameters to zero.
See also \cite[Example 3.8]{QRSW4} for more details.
\end{example}

\begin{example} \label{ex:torus}
The previous example could easily be extended to a computation of the invariant of the $(2,n)$-torus link $T_{2,n}$ using \cite{Roz}.

For odd $n$,
\[
\CKh^{\dif_c}(T_{2,n}) \cong
\begin{cases}
q^{-n-2} U_1 \oplus \bigoplus_{j=1}^{\frac{n-1}{2}} 
\left( 
t^{-2j} q^{-n+4j-2}U_0 \oplus t^{-2j-1} q^{-n+4j+2}U_0
\right) \text{ if } p \nmid n \\
q^{-n-2} (U_0 \oplus q^2 U_0) \oplus \bigoplus_{j=1}^{\frac{n-1}{2}} ,
\left( 
t^{-2j} q^{-n+4j-2}U_0 \oplus t^{-2j-1} q^{-n+4j+2}U_0
\right) \text{ if } p | n .
\end{cases}
\]

For even $n$,
\[
\CKh^{\dif_c}(T_{2,n}) \cong
\begin{cases}
q^{-n-2} U_1 \oplus \bigoplus_{j=1}^{\frac{n-2}{2}} 
\left( 
t^{-2j} q^{-n+4j-2}U_0 \oplus t^{-2j-1} q^{-n+4j+2}U_0
\right) \oplus t^{-n} q^{n-2} (U_0 \oplus q^2 U_0) \text{ if } p \nmid n , \\
q^{-n-2} (U_0 \oplus q^2 U_0) \oplus \bigoplus_{j=1}^{\frac{n-2}{2}} 
\left( 
t^{-2j} q^{-n+4j-2}U_0 \oplus t^{-2j-1} q^{-n+4j+2}U_0
\right) \oplus t^{-n} q^{n-2} (U_0 \oplus q^2 U_0) \text{ if } p | n .
\end{cases}
\]

\end{example}

\begin{example}
Let $\ovoid^{\pm 1}$ be the $\pm 1$ framed unknots.  
\[
\CKh^{\dif_c}(\ovoid^{+1}) \cong q^{-3}U_1, \quad \quad
\CKh^{\dif_c}(\ovoid^{-1}) \cong q^{}U_1 .
\]
Thus 
\[
 \CKh^{\dif_c}(\bigsqcup_n \ovoid^{+1}) \cong q^{-3n}U_1^{\otimes n}, \quad \quad
\CKh^{\dif_c}(\bigsqcup_n \ovoid^{-1}) \cong q^{n}U_1^{\otimes n} .
\]
Then by \cite[Equation 8]{KQ}, any $p$-complex can arise as the homology of a link.
\end{example}

\begin{example} \label{ex:+1framed2unlink}
When $L$ is the  $2$-cable of the $+1$-framed unknot, it becomes the Hopf link in which each component has framing number 1. The Khovanov chain complex of $L$ in the relative homotopy category is given by
\[
\CKh^{\dif_c}(L) \cong 
q^{-4}(U_0 \oplus q^2 U_0) \oplus t^2 q^{-8} U_1.
\]
This follows from the Hopf link calculation in Example  \ref{eg-hopf-link}, while taking into account that the framings on each component introduce grading and homological degree shifts, which actually do not change the $H$-module structure.
\end{example}

\begin{rem}

The Jones polynomial at a sixth root of unity is identically 1.
However, at $p=3$ the homology is a non-trivial invariant.

\end{rem}

\subsection{Functoriality}
Ehrig, Tubbenhauer and Wedrich \cite{ETW} have proved that Khovanov--Rozansky $\mathfrak{gl}_N$ homology enjoys full functoriality with respect to link cobordisms. This extends earlier work on functoriality for the $\mathfrak{sl}(2)$ case in \cite{BHPW, blan1, Cap, CMW, EST1} and \cite{Vog}.  Functoriality for the $\mathfrak{sl}(2)$ case was first established by Jacobsson \cite{Jac} and Khovanov \cite{Khcob} who proved it up to a sign.

\begin{thm} 
  The Khovanov chain complex endowed with the $p$-differential $\dif_c$ provides a functor $\CKh^{\dif_c}(\cdot)$ from the category of framed link cobordisms to the derived category of $H$-modules. 
  
  Consequently, the slash Khovanov homology $\Kh_/^{\dif_c}(\cdot)~$ is a functor from the category of framed link cobordisms to the stable category $H\udmod$.
\end{thm}

\begin{proof}
This follows from the functoriality of Khovanov homology and from the fact that $\dif_c$ annihilates cup, cap, and saddle foams.
\end{proof}

\section{Colored link homology at a root of unity}
In this section we extend various finite-dimensional categorifications of the colored Jones polynomial to the $p$-DG setting. Let us immediately note that the isomorphisms established by \cite{VM} between the various categorifications remain true as long as the color on each component is less than $p$.

\subsection{Colored Khovanov homology}

\subsubsection{Tensor resolutions}
We review Khovanov's categorification of the colored Jones polynomial which was extended by Beliakova--Wehrli to a theory over $\mathbb{Z}[\frac{1}{2}]$.  
Thus, we again assume $\Bbbk$ is a field of characteristic $p>2$.

Let ${\bf n}=(n_1,\ldots,n_r)$ be a sequence of natural numbers.  Let $L_{\bf n}$ be a framed oriented link with $r$ components with the $i$th component labeled by the $(n_i+1)$-dimensional irreducible representation of quantum $\mathfrak{sl}(2)$.

Let $L^{\bf n}$ be a cabling of $L$ where the $i$th component becomes cabled by $n_i$ parallel strands and the orientations alternate between the original one and its opposite.

For an array of $n$ dots, a $k$-pairing is a selection of $k$ disjoint adjacent pairs of dots.  
Similarly, define a ${\bf k}$-pairing of ${\bf n}$ dots.
Let $\Gamma_{\bf n}$ be a graph whose vertices are ${\bf k}$-pairings.  There is an edge from one pairing ${\bf k}$ to another ${\bf k}'$ if ${\bf k}'$ can be obtained from ${\bf k}$ by adding a single pairing.

For example, the graph $\Gamma_{(2,2)}$ is given in \eqref{eqn-Gamma22}.
\begin{equation}\label{eqn-Gamma22}
\begin{gathered}
\xymatrix@=0.75em{
& & & \bullet  \ar@{-}[r] & \bullet \\
& & & \bullet & \bullet \ar[drr]^{} \\
\bullet & \bullet \ar[urr]^{}  & & & & & \bullet   \ar@{-}[r]& \bullet    \\
\bullet & \bullet \ar[drr]^{} & & & & & \bullet  \ar@{-}[r] & \bullet \\
& & & \bullet  & \bullet   \ar[urr]^{} \\
& & & \bullet \ar@{-}[r] & \bullet \\
}
\end{gathered}
\end{equation}

To a ${\bf k}$-pairing and a colored link $L_{\bf n}$, associate the Khovanov chain complex $\CKh^{\dif_c}(L^{{\bf n}-2{\bf k}})$ equipped with a $p$-differential $\dif_c$. If there is an edge from ${\bf k}$ to ${\bf k}'$, by the functoriality of Khovanov chain complexes, there is a map from $\CKh^{\dif_c}(L^{{\bf n}-2{\bf k}})$ to
$\CKh^{\dif_c}(L^{{\bf n}-2{\bf k}'})$ given by gluing an annulus \eqref{eq:annulus1} between the strands of the cable which form a pair in ${\bf k}'$ but not in ${\bf k}$.
Note that the map associated to this annulus commutes with the $p$-differential.
\begin{equation} \label{eq:annulus1}
  \NB{\tikz[font= \tiny,
  scale=0.6]{\begin{scope}[scale=1.5]
  \draw  (2, 0) arc (0:-180:2cm and 0.5cm);
  \draw[densely dashed] (2,0) arc (0:180:2cm and 0.5cm);
  \draw[densely dashed] (0,0) circle (1cm and 0.25cm);
  \draw (1.5,1) arc (0:-180: 1.5cm and 0.375cm);
  \draw (1.5,1) arc (0:180: 1.5cm and 0.375cm);
  \begin{scope}[very thin]
  \draw (1.5, 1) .. controls +(0.3, 0) and +(0, 0.3) .. (2, 0)
  coordinate[pos=0.6] (a) node[pos=0.35, sloped, font=\small, below] {}; 
  \draw (-1.5, 1) .. controls +(-0.3, 0) and +(0, 0.3) .. (-2, 0)
 coordinate[pos=0.6] (b) node[pos=0.35, sloped, font=\small, below] {}; 
  \draw[densely dashed] (1.5, 1) .. controls +(-0.3, 0) and +(0, 0.3) .. (1, 0);
  \draw[densely dashed] (-1.5, 1) .. controls +(0.3, 0) and
  +(0, 0.3) .. (-1, 0); 
  \end{scope}
\end{scope}}} 
\end{equation}

\begin{rem}
 Instead of using alternating orientations on cabled strands, one may align the orientations, but use certain singular $\mathfrak{gl}(2)$ foams as cobordisms.  The following diagram represents a morphism from two nested unknots $\sqcup_2 \ovoid$ into 
 the unknot colored by the determinant representation of $\mathfrak{gl}(2)$,
given by the cylinder
\[
\NB{\tikz[font= \tiny,
  scale=0.6]{\begin{scope}[scale=1.5]
  \draw  (2, 0) arc (0:-180:2cm and 0.5cm);
  \draw[densely dashed] (2,0) arc (0:180:2cm and 0.5cm);
  \draw[densely dashed] (0,0) circle (1cm and 0.25cm);
  \draw[double] (1.5,1) arc (0:-180: 1.5cm and 0.275cm);
  \draw[densely dashed, double] (1.5,1) arc (0:180: 1.5cm and 0.275cm);
  \draw[double] (1.5,1.75) arc (0:-180: 1.5cm and 0.275cm);
    \draw[double] (1.5,1.75) arc (0:180: 1.5cm and 0.275cm);
  \begin{scope}[very thin]
  \draw (1.5, 1) .. controls +(0.3, 0) and +(0, 0.3) .. (2, 0)
  coordinate[pos=0.6] (a) node[pos=0.35, sloped, font=\small, below] {}; 
  \draw (-1.5, 1) .. controls +(-0.3, 0) and +(0, 0.3) .. (-2, 0)
 coordinate[pos=0.6] (b) node[pos=0.35, sloped, font=\small, below] {}; 
  \draw[densely dashed] (1.5, 1) .. controls +(-0.3, 0) and +(0, 0.3) .. (1, 0);
  \draw[densely dashed] (-1.5, 1) .. controls +(0.3, 0) and
  +(0, 0.3) .. (-1, 0); 
 \draw[double] (1.5, 1) -- +(0,0.75);
  \draw[double] (-1.5, 1) -- +(0,0.75);
  \end{scope}
\end{scope}}}
\]
This foam can be regarded as the product of the circle $\mathbb{S}^1$ with the web
\[
\NB{\tikz[font= \tiny,
  scale=0.6]{\begin{scope}
  \coordinate (bl) at (-0.5, -1);
  \coordinate (br) at ( 0.5, -1);
  \coordinate (bm) at (  0,-0.3);
  \coordinate (tl) at (-0.5,  1);
  \coordinate (tr) at ( 0.5,  1);
  \coordinate (tm) at (  0, 0.3);
  \draw[>-]  (bl) .. controls +( 0, 0.5) and +(0,0) .. (bm);
  \draw[>-]  (br) .. controls +( 0, 0.5) and +(0,0) .. (bm);

  \draw[double] [->] (bm) -- (tm); 
 
\end{scope}}}
\]
By equations \eqref{eq:e-act-dig-cap} and \eqref{eq:e-act-zip}, $\dif_c$ acts trivially on this foam. Similarly, equations \eqref{eq:e-act-dig-cup} and \eqref{eq:e-act-unzip} show that $\dif_c$ acts trivially on the foam turned upside-down.
\end{rem}

Now form a chain complex such that in homological degree $i$ the object is 
\[
\bigoplus_{\bf k} \CKh^{\dif_c}(L^{{\bf n}-2{\bf k}}), \text{where the sum is over pairings} \hspace{.05in} {\bf k} \hspace{.05in} \text{with} \hspace{.05in} i \hspace{.05in} \text{pairs}  .
\]
There is a map in the complex from $\CKh^{\dif_c}(L^{{\bf n}-2{\bf k}})$ to
$\CKh^{\dif_c}(L^{{\bf n}-2{\bf k}'})$ coming from the annulus if there is an edge from ${\bf k}$ to ${\bf k}'$ in the graph $\Gamma_{\bf n}$.  One also chooses a sign as prescribed by \cite{BW}.

In characteristic $p$, if one treats a dot in the dotted diagram as the $\mathfrak{sl}(2)$-representation $V_1 \cong \Bbbk^2$, then the resolution complex utilized by Khovanov becomes a resolution of the (dual) Weyl module. For instance, in characteristic 3, let us denote the natural projection map arising from the decomposition $V_1\otimes V_1\cong V_0\oplus V_2$: for $i,j\in \{1,-1\}$,
\[
\cap: V_1\otimes V_1\to V_0 , \quad v_{i}\otimes v_{j} \mapsto \delta_{i+j,0}v_0.\]
As in \cite{Khcolored}, the diagram 
    \begin{equation}\label{eqn-Gamma3}
\begin{gathered}
\xymatrix@C=0.25em{
& & & & \bullet \ar@{-}[r] & \bullet & \bullet \\
\bullet & \bullet & \bullet \ar[urr] \ar[drr] \\
& & & & \bullet & \bullet \ar@{-}[r] & \bullet \\
}
\end{gathered}
\end{equation}
translates into the complex
\[
V_1^{\otimes 3} \xrightarrow{(\cap_1,\cap_2)} V_1\oplus V_1,
\]
where $\cap_1$ and $\cap_2$ are given respectively by
\[
\cap_1: V_1^{\otimes 3}\cong (V_1\otimes V_1)\otimes V_1\xrightarrow{\cap\otimes \Id} V_0\otimes V_1\cong V_1,\quad \quad
\cap_2: V_1^{\otimes 3}\cong V_1\otimes( V_1\otimes V_1)\xrightarrow{\Id\otimes\iota }V_1\otimes V_0 \cong V_1   .
\]
As $V_1\otimes V_1\cong V_0\oplus V_2$, and $V_2$ is the Steinberg module, we obtain that
\[
\cap_1: V_1^{\otimes 3}\cong (V_1\otimes V_1)\otimes V_1 \cong (V_0\oplus V_2)\otimes V_1 \cong V_1 \oplus \projective{1} \to V_1
\]
\[
\cap_2: V_1^{\otimes 3}\cong V_1\otimes (V_1\otimes V_1) \cong V_1 \otimes (V_0\oplus V_2)  \cong  V_1 \oplus \projective{1}  \to V_1
\]
The direct sum of the maps $(\cap_1,\cap_2)$ is surjective, and its kernel is equal to the kernel of the natural projection from $P_1$ onto $V_1$. This is the Weyl module of highest weight 3: 
\begin{equation}\label{eqn-a-dual-Weyl-module}
\begin{gathered}
\xymatrix@C=0.75em{
\textrm{Weights:} && 3 && 1 && -1 && -3 \\
&& \bullet  \ar@/^/[rr]^{\fl} && \bullet \ar@/^/[rr]^{\fl}  && \bullet \ar@/^/[ll]^{\el}  && \bullet \ar@/^/[ll]^{\el}
}
\end{gathered}
\end{equation}
In this sense, we may regard Khovanov's colored homology theory in characteristic $p$ with colorings given by (dual) Weyl modules of $\mathfrak{sl}(2)$.

Denote the image of $\CKh^{\dif_c}(L_{\bf n})$ in $H\udmod$ by $\Kh^{\dif_c}_/(L_{\bf n})$.  The following theorem now follows immediately from \cite[Proposition 3]{Khcolored} and \cite[Theorem 1]{BW}.

\begin{thm}
Let ${\bf n}=(n_1,\ldots,n_r)$ with each $n_i \in \mathbb{Z}_{\geq 1}$.  
\begin{enumerate}
\item If $D_1$ and $D_2$ are two planar projections of a link $L$, then their associated colored complexes are isomorphic in the relative homotopy category. 
    \item The slash Khovanov homology $\Kh^{\dif_c}_/(L_{\bf n})$ is a framed link invariant valued in $H\udmod$.  
    \item For a link $L$, $\chi_{\mathbb{O}}(\Kh^{\dif_c}_/(L_{\bf n}))$ is the colored Jones polynomial evaluated at a $2p$th root of unity.
    \end{enumerate}
\end{thm}

\subsubsection{Symmetric invariants}
An equivalent way of constructing Khovanov's colored homology is to take the symmetric group invariant part of the cabled homology.
As we work in characteristic $p$,  we assume, in this version of colored homology, that the colors of the strands $n_i$ are less than the characteristic of the field ($n_i<p$).

Given a framed oriented link $L$ and a coloring of the strands ${\bf n}=(n_1,\ldots,n_r)$, let $L^{\bf n}$ be the associated cabled framed oriented link and $L_{\bf n}$ the associated colored link.

Khovanov \cite{Khcolored} conjectured that the braid group action on $\Kh(L^{\bf n})$ factors through the symmetric group $S_{\bf n}:=S_{n_1}\times \dots \times S_{n_r}$.  This was proved by Grigsby--Licata--Wehrli \cite{GLW}. The braid group action commutes with $\dif_c$, so that it descends to a $p$-differential, still denoted $\dif_c$, on the $S_{\bf n}$-invariant part of the total homology of the $\bf{n}$-cable of $L$.

\begin{defn}
Equip the total cohomology group $\mH(\CKh^{\dif_c}(L^{\bf n}))$ with the symmetric group $S_{\bf n}$ action.
Define $\SKh^{\dif_c}(L_{\bf n})$ to be the $S_{\bf n}$-invariants of $\mH(\CKh^{\dif_c}(L^{\bf n}))$ with the induced $H=\Bbbk[\dif_c]/(\dif_c^p)$ action arising from the $p$-differential $\dif_c$.
\end{defn}

One key drawback of this construction is that we do not know how the invariants under the symmetric group action can be taken on the Khovanov chain complexes of cabled links. 

\begin{thm}\label{thm:SKH}
Let ${\bf n}=(n_1,\ldots,n_r)$ with each $n_i<p$.  
\begin{enumerate}
\item If $D_1$ and $D_2$ are two planar projections of a link $L$, then their associated symmetrically colored homologies are isomorphic as bigraded $H$-modules.
    \item $\SKh^{\dif_c}(L_{\bf n})$ is a framed link invariant.  
    \item For a link $L$, $\SKh^{\dif_c}(L_{\bf n})$ is a chain complex of graded $H$-modules with the zero differential, and $\chi_{\mathbb{O}}(\SKh^{\dif_c}(L_{\bf n}))$ is the colored Jones polynomial evaluated at a $2p$th root of unity.
    \end{enumerate}
\end{thm}

\begin{proof}
The fact that the homology without the $p$-differential is isomorphic to Khovanov's original colored theory $ \Kh(L_{\bf n})$ follows from \cite[Proposition 3.18]{VM} using \cite[Theorems 4.8, 4.10, 4.11]{BHPW2}.  Note that they work over a field of characteristic 0, but the proofs work assuming all the colors $n_i$ are less than $p$ \cite[Remark 3.17]{VM}.

Since the $p$-differential commutes with the symmetric group action, we immediately have the corresponding statements in the $p$-DG context.
\end{proof}

\subsubsection{Tilting idempotents}
We next seek to present colored homology theories using idempotents defined in Section \ref{sec:TLidemp}. To do so, we first recall some general setup concerning idempotents in homotopy categories. 

Let $M^\bullet$ be a bounded chain complex in the homotopy category $\mc{C}^b(\mc{A})$ of an abelian category $\mc{A}$, and let $e:M^\bullet \to M^\bullet$ be an idempotent endomorphism of $M^\bullet$ in the homotopy category. We would like to know whether there exist objects $e\cdot M^\bullet$, $(1-e)\cdot M^\bullet$ in $\mc{C}^b(\mc{A})$ such that
\[
M^\bullet \cong e\cdot M^\bullet \oplus (1-e)\cdot M^\bullet.
\]
The problem of existence and uniqueness of such a direct sum decomposition in the homotopy category is answered by the following result of Balmer and Schlichting.

\begin{thm}[Balmer--Schlichting]
\label{thm:BaSch}Let $\mc{A}$ be an abelian category. Then the bounded homotopy category $\mc{C}^b(\mc{A})$ is idempotent complete. Furthermore, any object representing the image of an idempotent is unique up to homotopy equivalence.
\end{thm}
\begin{proof}
    In \cite[Theorem 2.8]{BaSch}, Balmer and Schlichting prove that the bounded derived category of an exact category is idempotent complete. To apply their result for homotopy categories, one just needs the abelian category $\mc{A}$ to be equipped with the exact structure where admissible exact sequences are split short exact sequences in $\mc{A}$. With respect to this exact structure, the bounded derived category of the exact category $\mc{A}$ agrees with the usual bounded homotopy category $\mc{C}^b(\mc{A})$. 

    The second statement follows from the Yoneda Lemma. More precisely, if $M^\bullet \in \mathcal{C}^b(\mathcal{A})$, and $e\in \End_{\mc{C}^b(\mc{A})}(M^\bullet)$ is an idempotent, then an image object $I^\bullet\in \mc{C}^b(\mc{A})$ of $e$ is represented by the datum
    \[ \iota: I^\bullet \to M^\bullet, \quad \jmath: M^\bullet \to I^\bullet \]
    satisfying $\jmath \circ \iota \simeq \Id_{I}$, $\iota\circ \jmath\simeq e$. If 
    \[ \iota^\prime: J^\bullet \to M^\bullet, \quad \jmath^\prime: M^\bullet \to J^\bullet \]
    is another such datum, then $\jmath^\prime \circ \iota$ and $\jmath \circ \iota^\prime$ are clearly homotopy equivalences that are inverses of each other.
\end{proof}

Now let $L_{//}$ denote the link $L$ where each component $L_i$ has been replaced by its $(n_i)$-anti-parallel cabling. 

For any integer $n \geq 0 $, $E_n \times \mathbb{S}^1$ is a formal linear combination of cobordisms from the $n$-anti-parallel cabling of the unknot to itself. Similarly, if $K$ is an oriented framed knot (or a component of a link), one can construct a cobordism $E_n(K)$ from its $n$-anti-parallel cabling to itself. This cobordism induces an endomorphism of the Khovanov chain complex of the $n$-anti-parallel cabling to itself. This turns out to be an idempotent. Idempotents on different components of a link commute. 

The \emph{tilting colored Khovanov} chain complex, well-defined up to homotopy thanks to Theorem \ref{thm:BaSch}, of a link $L=(L_i)$ colored by ${\bf n}=(n_i)$ is the image of $\CKh(L_{//})$ under the composition of these idempotents (which is itself an idempotent denoted $E_{\bf n}(L)$).  Denote this image by $\TKh(L_{\bf n})$. 
The idempotent $E_{\bf n}$ commutes with $\dif_c$ so that one can consider $\TKh^{\dif_c}(L_{\bf n})$ and its image $\TKh^{\dif_c}_/(L_{\bf n})$ in $H\udmod$ (slash cohomology). As in Theorem~\ref{thm:SKH}, one has:

\begin{thm}
Let ${\bf n}=(n_1,\ldots,n_r)$. 
\begin{enumerate}
\item If $D_1$ and $D_2$ are two planar projections of a link $L$, then their associated tilting colored complexes are isomorphic in the relative homotopy category.
    \item $\TKh^{\dif_c}_/(L_{\bf n})$ is a framed link invariant.  
        \item $\chi_{\mathbb{O}}(\TKh_/^{\dif_c}(L_{\bf n}))$ is the colored Jones polynomial evaluated at a $2p$th root of unity, when all $n_i < p$.
    \item $\chi_{\mathbb{O}}(\TKh_/^{\dif_c}(L_{\bf n}))\equiv 0$, when all $n_i \geq p-1$.
\end{enumerate}
\end{thm}

\begin{rem}
Note that in \cite{VM}, the homology of the colored Khovanov complex $\TKh(L_{\bf n})$ is denoted by \emph{(Im)}. When all the coloring labels are less than $p$, there is an isomorphism of the usual homology $\mH(\TKh(L_{\bf n})) \cong \SKh^{}(L_{\bf n}) $.
\end{rem}

\subsection{Functoriality}

\begin{thm}
  The tilting colored Khovanov homology $\TKh$ induces a functor from the category of colored framed link cobordisms to that of doubly graded vector spaces.  
\end{thm}

\begin{rem}
  The cabling operation seems hard to achieve functorially, so we cannot just say that Khovanov homology of the cabling is functorial as a composition of two functors. 
\end{rem}

\begin{proof}
Given an oriented, colored, and framed link diagram $D_{\bf n}$, one can associate (non-functorially) a diagram $D_{//}$ for its ${\bf n}$-anti-parallel cabling. In the homotopy category, the Khovanov chain complex of $D_{//}$ is well-defined by functoriality of Khovanov homology. Indeed, the small variations one may encounter when one cables diagrams or surfaces are harmless due to functoriality of Khovanov homology.

  For the same reasons, given a diagram $D_{\bf n}$, there is a well-defined projector $E_{\bf n}$ on $\CKh(D_{//})$.

  For any elementary (oriented, colored, framed) cobordism  $\Sigma \colon D^1_{\bf n^1} \to D^2_{\bf n^2}$, define 
\[
  \TKh(\Sigma)= E_{\bf n^2}(D^2) \circ \CKh(\Sigma_{//}) \circ  E_{\bf n^1}(D^1),
\]
where $\Sigma_{//}$ denotes the cabling of the surface $\Sigma$.
  We want to prove that different compositions of elementary cobordisms yielding the same cobordism $\Sigma \colon D^1 \to D^2$ give the same map from $\TKh(D^1) \to \TKh(D^2)$. Movie moves for framed oriented cobordisms are listed in \cite{BW}. We need to prove invariance under these movie moves.

  For all but one of them, invariance follows from the functoriality of Khovanov homology. Indeed, idempotents are given by cobordisms. By functoriality of Khovanov homology, one can move these cobordisms and for all these movie moves, one can actually push them to the first and last diagrams (which are the same for the two sides of the movie moves). Once this is done, functoriality of the uncolored Khovanov homology on the cabling of the movie move implies invariance. 

  \begin{figure}[ht]
    \centering
    \includegraphics[width=10cm]{\imagesfolder/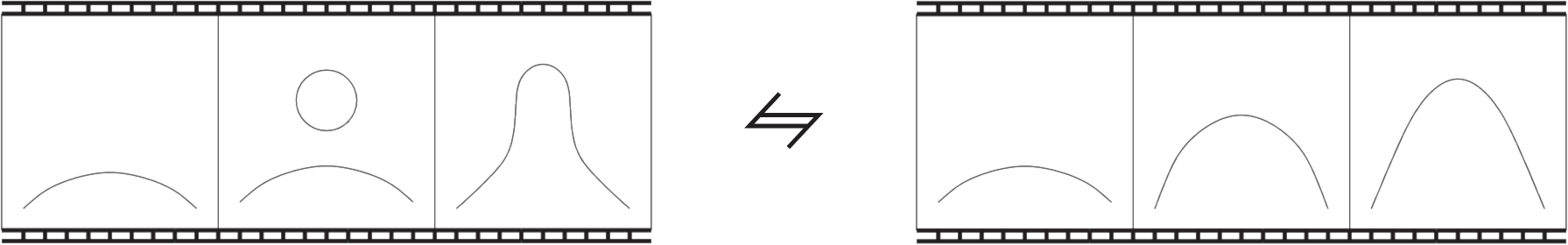}
    \caption{The movie move MM11}
    \label{fig:MM11}
  \end{figure}

  For the movie move MM11 (see Figure~\ref{fig:MM11}), this does not work since one cannot push the idempotent appearing in the middle frame on the circle all the way through the top or the bottom (because the circle disappears).

However, invariance does hold for other reasons, as we shall now explain.
In the linear combination that defines the idempotent, there is one copy of the identity and a linear combination of terms that contain $\cap \times \mathbb{S}^1$. Composing such a term with two births (i.e., cups) forms an undotted sphere (see Figure~\ref{fig:spheres_idem}) which maps to zero by $\Kh$. Therefore, in this movie move, this idempotent acts as the identity and we can therefore replace it by the identity. Once this is done, we argue, as before, by functoriality of uncolored Khovanov homology on the cabling of the movie move MM11.\qedhere
\begin{figure}[ht]
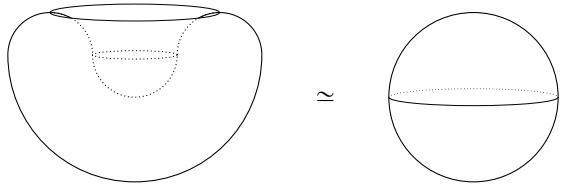

    \centering
    \NB{\tikz[scale=0.8]{\begin{scope}[scale=0.7, yscale=-1]
  \begin{scope}
    \draw (-1, 0) arc (180:300:1cm);
    \draw (-1, 0) arc (180:0:3cm);
    \draw (5, 0) arc (360:240:1cm);
    \draw[very thin] (0, -1) arc (180:360:2cm and -0.2cm);
    \draw[very thin] (0, -1) arc (180:0:2cm and -0.2cm);
  \end{scope}
  \begin{scope}[densely dotted]
    \draw (0, -1) arc (270:360:1cm);
    \draw (1, 0) arc (180:0:1cm);
    \draw (3, 0) arc (180:270:1cm);    
    \draw[very thin] (2,0) circle (1cm and 0.1cm);    
  \end{scope}

\node at (6.5, 1) {$\simeq$};

\begin{scope}[xshift=10cm, yshift=1cm]
  \draw (0,0) circle (2cm);
  \draw[very thin] (-2, 0) arc (180:360:2cm and -0.2cm);
  \draw[densely dotted, very thin] (-2, 0) arc (180:0:2cm and -0.2cm);
\end{scope}

\end{scope}}}
    \caption{Undotted sphere appearing when precomposing an idempotent with cups. }
    \label{fig:spheres_idem}
  \end{figure}
\end{proof}

Since the vertical idempotents are linear combinations of cobordisms, one obtains the following corollary immediately. 
\begin{cor}
    Tilting colored Khovanov homologies endowed with the $p$-differentials $\dif_c$ provide functors from the category of colored framed link cobordisms to the stable category $H\udmod$.
\end{cor}

\subsection{Examples}
\begin{example}\label{eg-colunknot}
As an $H$-module, the chain complex of the $n$-colored unknot $\ovoid_n$, for $n<p$
is:
\begin{equation}
\CKh^{\dif_c}(\ovoid_n) \cong 
\bigoplus_{j=0}^{n} q^{n-2j} U_0.
\end{equation}
\end{example}

\begin{example}\label{eg-framed2unknot}
As an $H$-module, the chain complex of the $2$-colored framed unknot $\ovoid_2^{+1}$
is:
\begin{equation}
\CKh^{\dif_c}(\ovoid_2^{+1}) \cong 
t^2 q^{-4} \left(q^{-4} U_1 \oplus t^{-2} U_0 \right) .
\end{equation}
This follows from Example \ref{ex:+1framed2unlink}.
\end{example}

\appendix

\addcontentsline{toc}{section}{References}

\bibliographystyle{alphaurl}
\bibliography{extracted}

\begin{thebibliography}{BHPW23}

\bibitem[BHPW19]{BHPW}
A.~Beliakova, M.~Hogancamp, K.~Putyra, and S.~Wehrli.
\newblock On the functoriality of sl(2) tangle homology, 2019.
\newblock \href {https://arxiv.org/abs/1903.12194} {\path{arXiv:1903.12194}}.

\bibitem[BHPW23]{BHPW2}
A.~Beliakova, M.~Hogancamp, K.~Putyra, and S.~Wehrli.
\newblock On unification of colored annular sl(2) knot homology, 2023.
\newblock \href {https://arxiv.org/abs/2305.02977} {\path{arXiv:2305.02977}}.

\bibitem[Bla10]{blan1}
C.~Blanchet.
\newblock An oriented model for {K}hovanov homology.
\newblock {\em J. Knot Theory Ramifications}, 19:291--312, 2010.
\newblock \href {https://arxiv.org/abs/1405.7246} {\path{arXiv:1405.7246}},
  \href {https://doi.org/10.1142/S0218216510007863}
  {\path{doi:10.1142/S0218216510007863}}.

\bibitem[BLS19]{BLS}
G.~Burrull, N.~Libedinsky, and P.~Sentinelli.
\newblock {$p$-Jones--Wenzl idempotents}.
\newblock {\em Adv. Math.}, 352:246--264, 2019.
\newblock \href {https://doi.org/10.1016/j.aim.2019.06.005}
  {\path{doi:10.1016/j.aim.2019.06.005}}.

\bibitem[BS01]{BaSch}
P.~Balmer and M.~Schlichting.
\newblock Idempotent completion of triangulated categories.
\newblock {\em Journal of Algebra}, 236(2):819--834, 2001.
\newblock URL:
  \url{https://www.sciencedirect.com/science/article/pii/S0021869300985296},
  \href {https://doi.org/10.1006/jabr.2000.8529}
  {\path{doi:10.1006/jabr.2000.8529}}.

\bibitem[BW08]{BW}
A.~Beliakova and S.~Wehrli.
\newblock Categorification of the colored {J}ones polynomial and {R}asmussen
  invariant of links.
\newblock {\em The Canadian Journal of Mathematics}, 60:1240--1266, 2008.
\newblock \href {https://arxiv.org/abs/math/0510382}
  {\path{arXiv:math/0510382}}, \href {https://doi.org/10.4153/CJM-2008-053-1}
  {\path{doi:10.4153/CJM-2008-053-1}}.

\bibitem[Cap08]{Cap}
C.~Caprau.
\newblock An sl(2) tangle homology and seamed cobordisms.
\newblock {\em Alg. Geom. Topol.}, 8(2):729--756, 2008.
\newblock \href {https://arxiv.org/abs/math/0707.3051}
  {\path{arXiv:math/0707.3051}}, \href {https://doi.org/10.2140/agt.2008.8.729}
  {\path{doi:10.2140/agt.2008.8.729}}.

\bibitem[CMW09]{CMW}
D.~Clark, S.~Morrison, and K.~Walker.
\newblock Fixing the functoriality of {K}hovanov homology.
\newblock {\em Geom. Topol.}, 13(3):1499--1582, 2009.
\newblock \href {https://arxiv.org/abs/math/0701339}
  {\path{arXiv:math/0701339}}, \href
  {https://doi.org/10.1142/S0218216521500747}
  {\path{doi:10.1142/S0218216521500747}}.

\bibitem[EST16]{EST1}
M.~Ehrig, C.~Stroppel, and D.~Tubbenhauer.
\newblock Generic $\mathfrak{gl}_2$-foams, web and arc algebras, 2016.
\newblock \href {https://arxiv.org/abs/1601.08010} {\path{arXiv:1601.08010}}.

\bibitem[ETW18]{ETW}
M.~Ehrig, D.~Tubbenhauer, and P.~Wedrich.
\newblock Functoriality of colored link homologies.
\newblock {\em Proc. Lond. Math. Soc. (3)}, 117(5):996--1040, 2018.
\newblock \href {https://arxiv.org/abs/1703.06691} {\path{arXiv:1703.06691}},
  \href {https://doi.org/10.1112/plms.12154} {\path{doi:10.1112/plms.12154}}.

\bibitem[GLW18]{GLW}
J.~E. Grigsby, A.~Licata, and S.~Wehrli.
\newblock Annular {K}hovanov homology and knotted {S}chur--{W}eyl
  representations.
\newblock {\em Comp. Math.}, 154:459--502, 2018.
\newblock \href {https://arxiv.org/abs/1505.04386} {\path{arXiv:1505.04386}},
  \href {https://doi.org/10.1112/S0010437X17007540}
  {\path{doi:10.1112/S0010437X17007540}}.

\bibitem[Hap88]{Hap88}
D.~Happel.
\newblock {\em Triangulated categories in the representation of
  finite-dimensional algebras}, volume 119 of {\em London Math. Soc. Lect. Note
  Ser.}
\newblock Cambridge University Press, Cambridge, 1988.
\newblock \href {https://doi.org/10.1090/S0273-0979-1990-15869-0}
  {\path{doi:10.1090/S0273-0979-1990-15869-0}}.

\bibitem[Jac04]{Jac}
M.~Jacobsson.
\newblock An invariant of link cobordisms from {K}hovanov homology.
\newblock {\em Algebr. Geom. Topol.}, 4(2):1211--1251, 2004.
\newblock \href {https://arxiv.org/abs/math/0206303}
  {\path{arXiv:math/0206303}}, \href {https://doi.org/10.2140/agt.2004.4.1211}
  {\path{doi:10.2140/agt.2004.4.1211}}.

\bibitem[Jan03]{Jantzen}
Jens~Carsten Jantzen.
\newblock {\em Representations of algebraic groups}, volume 107 of {\em
  Mathematical Surveys and Monographs}.
\newblock American Mathematical Society, Providence, RI, second edition, 2003.
\newblock \href {https://doi.org/10.1090/surv/107}
  {\path{doi:10.1090/surv/107}}.

\bibitem[Kho00]{KhJones}
M.~Khovanov.
\newblock A categorification of the {J}ones polynomial.
\newblock {\em Duke Math. J.}, 101(3):359--426, 2000.
\newblock \href {https://arxiv.org/abs/math/9908171}
  {\path{arXiv:math/9908171}}, \href
  {https://doi.org/10.1215/S0012-7094-00-10131-7}
  {\path{doi:10.1215/S0012-7094-00-10131-7}}.

\bibitem[Kho05]{Khcolored}
M.~Khovanov.
\newblock Categorifications of the colored {J}ones polynomial.
\newblock {\em J. Knot Theory Ramifications}, 14:111--130, 2005.
\newblock \href {https://arxiv.org/abs/math.QA/0302060}
  {\path{arXiv:math.QA/0302060}}, \href
  {https://doi.org/10.1142/S0218216505003750}
  {\path{doi:10.1142/S0218216505003750}}.

\bibitem[Kho06]{Khcob}
M.~Khovanov.
\newblock An invariant of tangle cobordisms.
\newblock {\em Trans. Amer. Math. Soc.}, 358:315--327, 2006.
\newblock \href {https://arxiv.org/abs/QA/0207264} {\path{arXiv:QA/0207264}},
  \href {https://doi.org/10.1090/S0002-9947-05-03665-2}
  {\path{doi:10.1090/S0002-9947-05-03665-2}}.

\bibitem[Kho16]{Hopforoots}
M.~Khovanov.
\newblock {H}opfological algebra and categorification at a root of unity: {T}he
  first steps.
\newblock {\em J. Knot Theory Ramifications}, 25(3):359--426, 2016.
\newblock \href {https://arxiv.org/abs/math/0509083}
  {\path{arXiv:math/0509083}}, \href
  {https://doi.org/10.1142/S021821651640006X}
  {\path{doi:10.1142/S021821651640006X}}.

\bibitem[KQ15]{KQ}
M.~Khovanov and Y.~Qi.
\newblock An approach to categorification of some small quantum groups.
\newblock {\em Quantum Topol.}, 6(2):185--311, 2015.
\newblock \href {https://arxiv.org/abs/1208.0616} {\path{arXiv:1208.0616}},
  \href {https://doi.org/10.4171/QT/63} {\path{doi:10.4171/QT/63}}.

\bibitem[KR16]{KRWitt}
M.~Khovanov and L.~Rozansky.
\newblock Positive half of the {W}itt algebra acts on triply graded link
  homology.
\newblock {\em Quantum Topol.}, 7(4):737--795, 2016.
\newblock \href {https://arxiv.org/abs/1305.1642} {\path{arXiv:1305.1642}},
  \href {https://doi.org/10.4171/QT/84} {\path{doi:10.4171/QT/84}}.

\bibitem[LS69]{LaSw}
R.~G. Larson and M.~E. Sweedler.
\newblock An associative orthogonal bilinear form for {H}opf algebras.
\newblock {\em Amer. J. Math.}, 91:75--94, 1969.
\newblock \href {https://doi.org/10.2307/2373270} {\path{doi:10.2307/2373270}}.

\bibitem[Qi14]{QYHopf}
Y.~Qi.
\newblock Hopfological algebra.
\newblock {\em Compos. Math.}, 150(01):1--45, 2014.
\newblock \href {https://arxiv.org/abs/1205.1814} {\path{arXiv:1205.1814}},
  \href {https://doi.org/10.1112/S0010437X13007380}
  {\path{doi:10.1112/S0010437X13007380}}.

\bibitem[QRSW21]{QRSW1}
Y.~Qi, L.-H. Robert, J.~Sussan, and E.~Wagner.
\newblock A categorification of the colored {J}ones polynomial at a root of
  unity, 2021.
\newblock \href {https://arxiv.org/abs/2111.13195} {\path{arXiv:2111.13195}}.

\bibitem[QRSW23]{QRSW3}
Y.~Qi, L.-H. Robert, J.~Sussan, and E.~Wagner.
\newblock Symmetries of equivariant {K}hovanov-{R}ozansky homology, 2023.
\newblock \href {https://arxiv.org/abs/2306.10729} {\path{arXiv:2306.10729}}.

\bibitem[QRSW24]{QRSW2}
Y.~Qi, L.-H. Robert, J.~Sussan, and E.~Wagner.
\newblock Symmetries of $\mathfrak{gl}_{N}$-foams.
\newblock {\em Quantum Topol.}, online first, 2024.
\newblock \href {https://arxiv.org/abs/2212.10106} {\path{arXiv:2212.10106}},
  \href {https://doi.org/10.4171/QT/215} {\path{doi:10.4171/QT/215}}.

\bibitem[QRSW25]{QRSW4}
Y.~Qi, L.-H. Robert, J.~Sussan, and E.~Wagner.
\newblock Remarks on some infinitesimal symmetries of {K}hovanov--{R}ozansky
  homologies in finite characteristic.
\newblock {\em Bulletin of the LMS}, online first, 2025.
\newblock \href {https://arxiv.org/abs/2412.08401} {\path{arXiv:2412.08401}},
  \href {https://doi.org/10.1112/blms.70172} {\path{doi:10.1112/blms.70172}}.

\bibitem[QS22]{QiSussanLink}
Y.~Qi and J.~Sussan.
\newblock On some {$p$}-differential graded link homologies.
\newblock {\em Forum Math. Pi}, 10:Paper No. e26, 58, 2022.
\newblock \href {https://arxiv.org/abs/2009.06498} {\path{arXiv:2009.06498}},
  \href {https://doi.org/10.1017/fmp.2022.19} {\path{doi:10.1017/fmp.2022.19}}.

\bibitem[Ric89]{Ri}
J.~Rickard.
\newblock Derived categories and stable equivalence.
\newblock {\em J. Pure Appl. Algebra}, 61(3):303--317, 1989.

\bibitem[Roz23]{Roz}
F.~Roz.
\newblock An $\mathfrak{sl}_2$ action on link homology of ${T}(2,k)$ torus
  links, 2023.
\newblock \href {https://arxiv.org/abs/2307.01910} {\path{arXiv:2307.01910}}.

\bibitem[STWZ23]{SuttonTubbenhauerWedrichZhu}
L.~Sutton, D.~Tubbenhauer, P.~Wedrich, and J.~Zhu.
\newblock {${\rm SL}_2$} tilting modules in the mixed case.
\newblock {\em Selecta Math. (N.S.)}, 29(3):Paper No. 39, 40, 2023.
\newblock \href {https://doi.org/10.1007/s00029-023-00835-0}
  {\path{doi:10.1007/s00029-023-00835-0}}.

\bibitem[vM25]{VM}
K.R. von Merkl.
\newblock Computing colored {K}hovanov homology, 2025.
\newblock \href {https://arxiv.org/abs/2505.03916} {\path{arXiv:2505.03916}}.

\bibitem[Vog20]{Vog}
P.~Vogel.
\newblock Functoriality of {K}hovanov homology.
\newblock {\em J. Knot Theory Ramifications}, 29(4), 2020.
\newblock \href {https://arxiv.org/abs/1505.04545} {\path{arXiv:1505.04545}},
  \href {https://doi.org/10.1142/S0218216520500200}
  {\path{doi:10.1142/S0218216520500200}}.

\end{thebibliography}

\noindent Y.~Q.: { \sl \small Department of Mathematics, University of Virginia, Charlottesville, VA 22904, USA} \newline \noindent {\tt \small email: yq2dw@virginia.edu}

\vspace{0.1in}

\noindent L.-H.~R.:  {\sl \small Université Clermont Auvergne, LMBP,
    Campus des Cézeaux, 3 place Vasarely, TSA 60026, CS 60026, 63178
    Aubière Cedex, France}
\newline \noindent {\tt \small email: louis-hadrien.robert@uca.fr  }

\vspace{0.1in}

\noindent J.~S.:  {\sl \small Department of Mathematics, CUNY Medgar Evers, Brooklyn, NY, 11225, USA}\newline \noindent {\tt \small email: jsussan@mec.cuny.edu \newline 
\sl \small Mathematics Program, The Graduate Center, CUNY, New York, NY, 10016, USA}

\vspace{0.1in}

\noindent E.~W.: { \sl \small Univ Paris Cité, IMJ-PRG, Univ Paris Sorbonne, UMR 7586 CNRS, F-75013, Paris, France} \newline \noindent {\tt \small email: wagner@imj-prg.fr}

\end{document}